\documentclass[final,3p,times]{elsarticle}

\usepackage{graphicx} 
\usepackage{caption}
\usepackage{subcaption}
\usepackage{amsmath, amsthm}
\usepackage{amsfonts}
\usepackage{import}
\usepackage{bbm}
\usepackage{algorithm}
\usepackage{algpseudocode}
\usepackage{comment}
\usepackage{siunitx}
\usepackage{pdfpages}
\usepackage{mathtools}
\usepackage{multicol}
\usepackage{bm}
\usepackage{hyperref}
\usepackage{tikz}
\usepackage{xcolor}
\usepackage{colortbl}
\usepackage{booktabs}
\usepackage[most]{tcolorbox}
\usepackage{pdfcomment}
\usepackage{dirtytalk}
\usepackage{algpseudocode}
\usepackage[utf8]{inputenc}
\usepackage{cleveref}

\newcommand{\BO}{\mathcal{O}}

\newcommand{\EX}{\mathbb{E}}

\newcommand{\St}{\mathcal{K}_{\Delta t}}
\newcommand{\Zt}{\mathcal{S}_{\Delta t}}
\newcommand{\fe}[1]{f(t^{(#1)})} 
\newcommand{\fp}[1]{f^{(#1)}} 
\newcommand{\dt}{\Delta t} 

\newcommand{\intv}{\int_{-\infty}^\infty}
\newtheorem{definition}{Definition}[section]
\newtheorem{theorem}{Theorem}[section]
\newtheorem{remark}{Remark}

\makeatletter
\def\ps@pprintTitle{%
	\let\@oddhead\@empty
	\let\@evenhead\@empty
	\let\@oddfoot\@empty
	\let\@evenfoot\@oddfoot
}
\makeatother

\begin{document}

\begin{frontmatter}



\title{Analysis of kinetic-diffusion Monte Carlo simulation and source term estimation scheme in nuclear fusion applications}


\author{Zhirui Tang\corref{cor1}\fnref{lab1}} 
\author{Julian Koellermeier\fnref{lab2,lab3}}
\author{Emil Løvbak\fnref{lab4}\corref{cor2}}
\author{Giovanni Samaey\fnref{lab1}\corref{cor2}}

\cortext[cor1]{Corresponding author: zhirui.tang@kuleuven.be} 
\cortext[cor2]{These authors contributed equally and share senior authorship.}
\fntext[extended]{This paper significantly extends a previously submitted short article by Tang et~al.\ (2025), which reported preliminary numerical results that are used in Section 5.6 of this work.}
\affiliation[lab1]{organization={Department of Computer Science, KU Leuven},
            city={Leuven},
            postcode={3001}, 
            country={Belgium}}

\affiliation[lab2]{organization={Bernoulli Institute, University of Groningen},
            city={Groningen},
            postcode={9747}, 
            country={The Netherlands}}
            
\affiliation[lab3]{organization={Department of Mathematics, Computer Science and Statistics,  Ghent University},
            city={Gent},
            postcode={9000}, 
            country={Belgium}}

\affiliation[lab4]{organization={Scientific Computing Center, Karlsruhe Institute of Technology},
	city={Karlsruhe},
	postcode={76131}, 
	country={Germany}}

\begin{abstract}
In plasma edge simulations, the behavior of neutral particles is often described by a Boltzmann–BGK equation. Solving this kinetic equation and estimating the moments of its solution are essential tasks, typically carried out using Monte Carlo (MC) methods.
However, for large-sized reactors, like ITER and DEMO, high collision rates lead to a substantial computational cost.
To accelerate the calculation, an asymptotic-preserving kinetic-diffusion Monte Carlo (KDMC) simulation method (Mortier et al., SIAM J. Sci. Comput., 2022) and a corresponding fluid estimation technique (Mortier et al., Contrib. Plasma Phys., 2022) have recently been proposed.  
In this work, we present a comprehensive analysis of the convergence of KDMC combined with the associated fluid estimation. The analysis consists of proving theoretical upper bounds for both KDMC and the fluid estimation, and numerical verifications of these bounds. 
In addition, we compare the analyzed algorithm with a purely fluid-based method using the fully kinetic MC method as a reference.
The algorithm consistently achieves lower error than the fluid-based method, and even one order of magnitude lower in a fusion-relevant test case. 
Moreover, the algorithm exhibits a significant speedup compared to the reference kinetic MC method. 
Overall, our analysis confirms the effectiveness of KDMC with the associated fluid estimation in nuclear fusion applications.

\end{abstract}




\end{frontmatter}


\section{Introduction}

Kinetic equations are widely used to model the behavior of particles in various physical systems, such as radiation transport \cite{vassilievMonteCarloMethods2017a}, rarefied gas dynamics \cite{cercignaniRarefiedGasDynamics2000}, and nuclear fusion \cite{borodinFluidKineticHybrid2022a, stangebyPlasmaBoundaryMagnetic2000b}. 
In these applications, the collision rate of particles may be high, requiring a huge amount of computational resources.  
In this case, approximate fluid models can be used to speed up the simulation without too much loss of accuracy \cite{othmerDiffusionLimitTransport2000b, maesHilbertExpansionBased2023a}. 
Nevertheless, situations where the collision rates vary dramatically across the simulation domain result in the fluid model being inaccurate in low-collision regimes, while the kinetic model remains inefficient in high-collision regimes. 

There are three main approaches for handling the above situation. The first one is domain decomposition \cite{densmoreHybridTransportdiffusionMethod2007a, crouseillesHybridKineticFluid2004a, degondFluidSimulationsLocalized2012}, where the simulation domain is divided into a kinetic and a fluid part, and different methods are applied to each part. 
A second approach, rather than decomposing the domain, is to decompose the solution itself into a kinetic part and a fluid part \cite{dimarcoHybridMultiscaleMethods2008b, horstenHybridFluidkineticModel2020a, aydemirUnifiedMonteCarlo1994}.
Both approaches face problems of how to decompose either the domain or the distribution function. The efficiency and accuracy of both approaches strongly depend on the quality of the decomposition.
A third approach is provided by the asymptotic-preserving method, avoiding the need for an explicit decomposition by employing a single adaptive method, see e.g., \cite{ gabettaRelaxationSchemesNonlinear1997b, bennouneUniformlyStableNumerical2008a, buetAsymptoticPreservingScheme2007, dimarcoHighOrderAsymptoticpreserving2012a}. 
Our main interest lies in asymptotic-preserving Monte Carlo (APMC) methods \cite{mortierKineticDiffusionAsymptoticPreservingMonte2022, pareschiTimeRelaxedMonte2001a, lovbakMultilevelMonteCarlo2021, lovbakAcceleratedSimulationBoltzmannBGK2023}. 
Using Monte Carlo (MC) enables APMC to be dimension-independent, facilitates handling complex geometries, and prevents the need to fully resolve the velocity domain. 
This work is based on the recently designed APMC method, known as kinetic-diffusion MC (KDMC) \cite{mortierKineticDiffusionAsymptoticPreservingMonte2022}.

The works mentioned above focus on the solution of the kinetic model, but this solution is not always the final quantity of interest. For instance, in plasma edge simulations, the primary focus is often placed on the first three moments of the particle distribution function, 
which describe the density of ionized particles, and the amount of transferred momentum and energy \cite{horstenComparisonFluidNeutral2016c, mortierAdvancedMonteCarlo2020}. 
KDMC approximates the solution of the kinetic model by generating discrete particle trajectories. Based on these trajectories, \cite{mortierEstimationPostprocessingStep2022a} proposes a fluid estimation procedure for estimating the moments.
As the main contribution in this work, we analyze KDMC and the associated fluid estimation procedure. 
The analysis is performed in one dimension for simplicity; however, all derivations remain valid in higher dimensions.

The analysis is divided into two parts.
The first part is on the convergence of the KDMC simulation, with the standard MC (also called the kinetc MC) method used as a benchmark.
The 1-Wasserstein distance is employed to compare the distributions generated by the standard MC and the KDMC simulations.
The original work \cite{mortierKineticDiffusionAsymptoticPreservingMonte2022} contains an analysis based on the assumption that the particle motion is conditioned by fixed parameters, and a sharp error bound is derived.
However, in practical simulations, these conditions are difficult to satisfy, making the convergence challenging to observe.
In this work, we propose an alternative analysis that does not rely on fixed parameter assumptions.
Although the resulting error bound is comparatively looser, the corresponding convergence is clearly observed in numerical experiments.
Furthermore, our analysis demonstrates a property of KDMC that the error is small in low- and high-collision regimes, and relatively large in the intermediate regime.

The second part addresses the estimation error. 
The estimation consists of the kinetic part estimated by the standard MC method and the diffusive part estimated by the fluid estimation procedure. 
The estimation error is first decomposed into the kinetic and the diffusive parts, each analyzed separately. The estimation error of the kinetic part, affected by parameters in KDMC, is purely statistical, as standard MC is unbiased.  
The estimation error of the diffusive part arises from multiple sources, including statistical and deterministic errors, and therefore a bias exists. 

In numerical experiments, we illustrate each error individually. 
Due to the design of KDMC, the errors impact each other.
Hence, we propose dedicated experiments to isolate these errors.
KDMC, which is based on the hybridization of the standard MC method and a fluid-based method, is expected to offer higher computational efficiency than the former while achieving greater accuracy than the latter. 
To demonstrate this, the algorithm is compared with the fluid-based method, using the standard MC as a reference. 
The comparison is carried out on a homogeneous background test case and a fusion-relevant heterogeneous test case. 
Subsequently, the computational efficiency of the algorithm is evaluated against that of the standard MC, and a speed-up rate is derived.

The remainder of the paper is organized as follows. The theoretical background, including the underlying kinetic equation, the fluid model, the KDMC algorithm, and the fluid estimation, is introduced in Section \ref{sec: preliminary theory}. 
The simulation error is discussed in Section \ref{sec: KDMC analysis}, followed by the estimation error in Section \ref{sec: estimation sec}. The numerical experiments are presented in Section \ref{sec: Num Ex}.
Finally, in Section \ref{sec: conclusion}, we conclude the work and give an outlook on future research. 

\section{KDMC simulation and fluid estimation}\label{sec: preliminary theory} 
This section introduces KDMC and the associated fluid estimation, together with the theoretical background required for the subsequent analysis. We begin in Section \ref{sec: BGK} with the kinetic equation governing neutral particles at the plasma edge.
Based on appropriate scaling assumptions, we introduce a fluid model (also known as the fluid limit) of the kinetic equation in Section \ref{sec: fluid model}. The KDMC method and the fluid estimation, both relying on this fluid model, are then described in Sections \ref{sec: KDMC} and \ref{sec: fluid estimaion}, respectively. 

\subsection{Boltzmann-BGK equation for neutral particles}\label{sec: BGK}
In nuclear fusion applications, neutral particles can be modeled using the steady-state Boltzmann equation with the BGK operator \cite{horstenHybridFluidkineticModel2020a}. 
In computer simulations, however, a time-dependent Boltzmann-BGK equation is used if a particle-based method (such as the standard MC or KDMC) is applied, 
since there is no steady-state behavior for a single particle, and the system evolves toward the steady-state from an initial state. 
In \ref{appendix: steady-state solution}, two time-dependent simulations, the terminal-time and the time-integrated simulations, are compared to validate the simulation of the steady-state system by the time-dependent model. 
In this work, we focus on the one-dimensional time-dependent Boltzmann-BGK equation describing the transport of neutrals driven by the collisions between neutrals and plasma, which reads
\begin{equation}\label{eqn: BBGK}
\partial_t f(x,v,t) + v\partial_x f(x,v, t) = R(x)\left(M_p(v|x)\intv f(x,v', t)dv'-f(x,v,t)\right),   
\end{equation}
with the initial condition $f(x,v,0)=S(x,v).$
Here, $f(x,v, t)$ is the distribution of the neutral particles at position $x$ with velocity $v$ at time $t$, and $S(x,v)$ is a neutral particle source from plasma recombination. 
We consider only the charge-exchange collision, in which neutrals are scattered and acquire new velocities. The collision rate is given as $R(x)$.
The normalized drifting Maxwellian distribution is
\begin{equation}\label{eqn: maxwellian}
    M_p(v|x) = \frac{1}{\sqrt{2\pi\sigma^2_p(x)}}\exp\left(-\frac{1}{2}\frac{(v-u_p(x))^2}{\sigma_p^2(x)}\right),
\end{equation}
a normal distribution with mean $u_p(x)$ and variance $\sigma^2_p(x)=T(x)/m$, where $u_p(x)$ and $T(x)$ are the mean velocity and the temperature of plasma, respectively. The constant $m$ is the mass of an ion.

Instead of the distribution $f(x,v,t)$, one is typically interested in the mass, momentum, and energy sources (from the neutrals to the plasma), 
in which the main components are three time-integrated moments, defined as
\begin{align}\label{eqn: moments}
    m_0(x) = \int_0^{\Bar{t}}\intv f(x,v,t)dvdt, \ \ m_1(x) &= \int_0^{\Bar{t}} \intv vf(x,v,t)dv dt, \ \ \text{and} \ \ m_2(x) = \int_0^{\Bar{t}}\intv \frac{v^2}{2}f(x,v,t)dvdt.
\end{align}
The time integration up to the final time $\Bar{t}=\infty$ is consistent with the time-integrated simulation, which, as shown in \ref{appendix: steady-state solution}, yields lower errors than the time-terminal simulation. 
This simulation also aligns with the strategy used in the EIRENE code \cite{reiterEIRENEB2EIRENECodes2005a} (see the “Unbiased Estimators” section of its user manual). 
Correspondingly, we adopt the time-integrated quantity throughout this work.
In what follows, we refer to the process of obtaining a discrete representation of the distribution $f(x,v,t)$ as simulation, while the calculation of moments \eqref{eqn: moments} is termed estimation.

\subsection{Diffusive scaling and fluid limit}\label{sec: fluid model}
In large-scale nuclear fusion devices, like ITER, the collision rate in certain regions of the plasma edge is relatively high, and the distribution of the neutral particles is close to equilibrium, which enables the approximation of the Boltzmann-BGK equation by a fluid model.

The fluid model can be derived by the diffusive scaling assumption with an asymptotic expansion \cite{othmerDiffusionLimitTransport2000b, maesHilbertExpansionBased2023a, horstenHybridFluidkineticModel2020a}. Specifically, we assume that the quantities of the simulation background scale with a small parameter $\varepsilon$, and the solution of \eqref{eqn: BBGK}
can be expanded as $f(x,v,t)= f_0(x,v,t) + \varepsilon f_1(x,v,t) + \varepsilon^2 f_2(x,v,t) + \cdots$.

    
In the so-called diffusive scaling, we assume that the velocity of an individual plasma particle, \( v_p \sim \mathcal{O}(1/\varepsilon) \), is high, while the mean plasma velocity, \( u_p(x) \sim \mathcal{O}(1) \), is relatively low. 
This difference leads to the velocity variance of the plasma  $\sigma_p^2(x) = \int (v_p-u_p(x))^2 M_p(v|x)\, dv \sim \mathcal{O}(1/\varepsilon^2)$. 
In the high-collisional regime, the neutral particles are close to the equilibrium of the plasma particles, so the velocity of a neutral particle $v$ has the same order of magnitude as $v_p$. 
In addition, collisions with the rate $R\sim\BO(1/\varepsilon^2)$ dominate in this regime, meaning that particles collide frequently. A detailed description can be found in \cite{maesHilbertExpansionBased2023a}.

Based on the above diffusive scaling assumptions, taking the first order truncated expansion $f_f(x,v,t)=f_0(x,v,t)+\varepsilon f_1(x,v,t)$, and applying the Hilbert expansion \cite{harrisIntroductionTheoryBoltzmann2004}, we obtain the fluid model 
\begin{equation}\label{eqn: fluid-model}
    \partial_t\rho(x,t) + \partial_x \left(u_p(x)\rho(x,t)\right)-\partial_x\left(\frac{1}{R(x)}\partial_x \left(\sigma_p^2(x)\rho(x,t)\right)\right)=0,
\end{equation}
where $\rho(x,t)=\intv f_f(x,v,t)dv$ is the density of particles. The full derivation can be found in \cite{maesHilbertExpansionBased2023a}.
Substituting the truncated expansion $f_f(x,v,t)$ into moments \eqref{eqn: moments}, we have
\begin{align}
\label{eqn: m0}
    m_0(x) &\approx \int_0^{\Bar{t}}\rho(x,t)dt, \\
\label{eqn: m1}
    m_1(x) &\approx \int_0^{\Bar{t}}\rho(x,t)u_p(x) - \frac{1}{R(x)}\partial_x\left(\sigma^2_p(x)\rho(x,t)\right)dt, \\
        \label{eqn: m2}    
    m_2(x) &\approx \int_0^{\Bar{t}}\frac{1}{2}\left(u_p^2(x)+\sigma_p^2(x)\right)\rho(x,t)- \frac{1}{R(x)}\partial_x \left(u_p(x)\sigma_p^2(x)\rho(x,t)\right)dt,
\end{align}
where $\rho(x,t)$ is the solution of \eqref{eqn: fluid-model}.
The approach that estimates the moments \eqref{eqn: moments} by solving the fluid model \eqref{eqn: fluid-model} and calculating \eqref{eqn: m0}–\eqref{eqn: m2} is referred to as the fluid-based method or simply the fluid method.
In Section \ref{sec: num. error of fluid model}, we demonstrate that the approximate moments \eqref{eqn: m0}–\eqref{eqn: m2} differ from the true moments \eqref{eqn: moments} by an error of order $\BO(\varepsilon^2)$. We call this error the model error.

\subsection{Kinetic-Diffusion Monte Carlo}\label{sec: KDMC}

In this section, we introduce the simulation methods used in this work for computing the particle distribution $f(x,v,t)$ described by the Boltzmann-BGK equation \eqref{eqn: BBGK}. 
We begin in Section \ref{sec: kinetic simulation} with the standard MC method, referred to as the kinetic simulation. 
In Section \ref{sec: diffusion simulation}, we present a diffusion simulation based on the fluid model \eqref{eqn: fluid-model}. 
Finally, in Section \ref{sec: kd}, we introduce KDMC, which hybridizes the kinetic and the diffusion simulations.



\subsubsection{Kinetic simulation}\label{sec: kinetic simulation}
With an MC method, the time-integrated particle distribution is approximated by summing Dirac measures associated with the discretized trajectory of particles.
To simulate the discretized trajectories, we consider the $i$-th particle with $i=1, 2, \cdots, I$ with weight $w_i$, where $I$ is the number of particles. 
The weight indicates the fraction of the total mass that the particle represents for the system. 
In particular, if the system has the initial density $\rho(x, t=0)$, the weight $w_i = \int \rho(x,t=0)dx/I$ for all $i$. 
Next, the simulation of this particle starts from the time $t=0$ at the position $x^0_i$ with the velocity $v^0_i$. 
After a flight time $\tau$ sampled from the exponential distribution $\text{Exp}(R(x^0_i))$, a charge-exchange collision occurs. 
At this point, the position of the particle updates to 
\begin{equation}\label{eqn: a kinetic step}
    x^1_i = x^0_i + \tau v^0_i,
\end{equation} 
and the particle is assigned a new velocity $v^1_i$ sampled from the Maxwellian distribution $M(v|x^1_i)$ given in \eqref{eqn: maxwellian}.
The movement \eqref{eqn: a kinetic step} is called a kinetic step.
Repeating the process up to the end of the simulation time $\bar{t}$, we obtain its trajectory denoted as $\{(x_i^k, v_i^k)\}_{k=0}^{K_i}$, where $K_i$ is the total number of states of the $i$-th particle.
The distribution obtained from the kinetic simulation in this section can then be assembled as in \eqref{eqn: K distribution}, using the trajectories of all particles.

\subsubsection{Diffusion simulation}\label{sec: diffusion simulation}
Starting from the fluid model \eqref{eqn: fluid-model}, we add and subtract the term $\partial_x\left(\partial_x(\frac{1}{R(x)})\sigma_P^2(x)\rho(x,t)\right)$ on the left-hand side, such that the fluid model becomes
\begin{equation}
    \partial_t \rho(x,t) + \partial_x\left(\left(u_p(x)+\partial_x\left(\frac{1}{R(x)}\right)\sigma_p^2(x)\right)\rho(x,t)\right) - \partial_{xx}\left(\frac{\sigma_P^2(x)}{R(x)}\rho(x,t)\right) = 0,
\end{equation}
a Fokker-Plank equation \cite{lapeyreIntroductionMonteCarlo2003a}. Its Ito stochastic differential equation (SDE) is
\begin{equation}\label{eqn: SDE of FK}
    dX_t = \left(u_p(X_t)+\partial_x\left(\frac{1}{R(X_t)}\right)\sigma_p^2(X_t)\right)dt + \sqrt{2\frac{\sigma_p^2(X_t)}{R(X_t)}}dW_t.
\end{equation}
with $W_t$ a Brownian motion. Let
$$
A(x) = u_p(x)+\partial_x\left(\frac{1}{R(x)}\right)\sigma_p^2(x) = u_p(x) + \frac{\partial D(x)}{\partial R(x)}\frac{\partial R(x)}{\partial x}, \quad \text{ and } \quad D(x) = \frac{\sigma_p^2(x)}{R(x)}. 
$$
Applying the Euler-Maruyama approximation,  this SDE indicates that the new position that the $i$-th particle travels a time $\dt$ from $x_i^k$ is 
\begin{equation}\label{eqn: SDE update}
    x_i^{k'} = x_i^k + A(x_i^k)\dt + \sqrt{2D(x_i^k)\dt}\xi,
\end{equation}
with $\xi\sim\mathcal{N}(0,1)$. 
The movement \eqref{eqn: SDE update} is called a diffusive step. 
When the scaling parameter $\varepsilon\rightarrow 0$, the true particle movement following Eq. \eqref{eqn: BBGK} converges to this step. 

\subsubsection{Kinetic-Diffusion simulation}\label{sec: kd}
The kinetic simulation described in Section~\ref{sec: kinetic simulation} is unbiased \cite{luxMonteCarloParticle2018a} but computationally expensive in highly collisional regimes. 
In contrast, the diffusion simulation introduced in Section~\ref{sec: diffusion simulation} offers a significant reduction in computational cost in these regimes, with only a minor loss in accuracy.
However, in low-collisional regimes, the fluid model \eqref{eqn: fluid-model} and thus the diffusion simulation no longer provides a good approximation, whereas the kinetic simulation is both valid and efficient.
In the intermediate regime, one must choose between accuracy and efficiency. 

These issues are handled in \cite{mortierKineticDiffusionAsymptoticPreservingMonte2022} by combining the advantages of the kinetic and the diffusion simulations. 
Specifically, we fix a time step size $\dt$. Consider the $i$-th particle. Within this time interval,
the particle executes first a kinetic step \eqref{eqn: a kinetic step} with time $\tau$ and then a diffusive step \eqref{eqn: SDE update} with $\theta=\dt-\tau$. If the flight time of the kinetic step $\tau > \dt$, no collision and diffusive step happens within the $\dt$. Only the kinetic step is performed.  
Intuitively, when the collision rate is high, i.e., $\tau \ll \dt$ on average, the particle will first perform a short kinetic step and then a dominant diffusive step for the rest of the time of size $\dt$. 
On the contrary, when the collision rate is low, i.e., $\tau \gg \dt$ on average, particles are primarily simulated by the kinetic step, and the influence of the inexact diffusive step is insignificant.  In both situations, at most two flights happen per time step of size $\dt$, and the simulation has only a small bias. 

Simulating the particle from $t=0$ to $\bar{t}=K\dt$, with $K$ the number of time steps, we obtain the trajectory of the $i$-th particle as 
\begin{equation}\label{eqn: trajectory}
    \{(x_i^0, v_i^0), ({x_i^0}', {v_i^0}'), (x_i^1, v_i^1), ({x_i^{1}}', {v_i^{1}}'), \cdots, (x_i^K, v_i^K), (x_i^{K'}, v_i^{K'}) \}.
\end{equation}
The state $({x_i^{k}}, {v_i^{k}})$ with the index $k$ is the starting state of the $k$-th kinetic step, 
and the state $({x_i^{k}}', {v_i^{k}}')$ with the index $k'$ is that of the $k$-th diffusive step, where $k=0,1, \ldots, K$. 
The weight $w_i = \int \rho(x,t=0)dx/I$  for all $i$ as stated in Section \ref{sec: kinetic simulation}. 
Repeating the above random process for all $I$ particles, the time-integrated distribution can be expressed as the Klimontovich distribution \cite{aydemirUnifiedMonteCarlo1994} 
\begin{align}\label{eqn: K distribution}
\begin{split}
     f(x) =\int_0^{\Bar{t}} \intv f(x,v,t)dv dt \approx \sum_{i=1}^I \sum_{k=0}^{K}w_i\delta(x-x_i^k) + \sum_{i=1}^I \sum_{k=0}^{K}w_i\delta(x-x_i^{k'}). 
\end{split}
\end{align}

The algorithm is illustrated graphically in Figure~\ref{fig: KDMC}. The corresponding pseudocode, including the fluid estimation procedure described in Section~\ref{sec: fluid estimaion}, is presented in Algorithm~\ref{alg: KDMC with f}. The analysis of KDMC is provided in Section~\ref{sec: KDMC analysis}.

\begin{figure}[h]
	\centering
	
	\tikzset{every picture/.style={line width=0.75pt}} 
	\begin{tikzpicture}[x=0.75pt,y=0.75pt,yscale=-1,xscale=1]
		
		\draw [color={rgb, 255:red, 0; green, 0; blue, 255 }  ,draw opacity=1 ]   (119.8,119.83) -- (137.29,119.83) ;
		\draw [color={rgb, 255:red, 0; green, 29; blue, 255 }  ,draw opacity=1 ] [dash pattern={on 0.84pt off 2.51pt}]  (137.29,119.83) -- (239.8,119.83) ;
		\draw [color={rgb, 255:red, 0; green, 29; blue, 255 }  ,draw opacity=1 ]   (239.8,119.83) -- (292.29,119.83) ;
		\draw [color={rgb, 255:red, 0; green, 29; blue, 255 }  ,draw opacity=1 ] [dash pattern={on 0.84pt off 2.51pt}]  (292.29,119.83) -- (359.8,119.83) ;
		\draw [color={rgb, 255:red, 0; green, 29; blue, 255 }  ,draw opacity=1 ]   (359.8,119.83) -- (479.8,119.83) ;
		\draw    (120.43,130.52) -- (479.43,130.52) ;
		\draw    (240.43,125.52) -- (240.43,136.52) ;
		\draw    (120.43,125.52) -- (120.43,136.52) ;
		\draw    (479.43,125.52) -- (479.43,136.52) ;
		\draw    (359.43,125.52) -- (359.43,136.52) ;
		
		\draw (409,135) node [anchor=north west][inner sep=0.75pt]  [font=\fontsize{0.71em}{0.85em}\selectfont]  {$\Delta t$};
		\draw (289,135) node [anchor=north west][inner sep=0.75pt]  [font=\fontsize{0.71em}{0.85em}\selectfont]  {$\Delta t$};
		\draw (171,135) node [anchor=north west][inner sep=0.75pt]  [font=\fontsize{0.71em}{0.85em}\selectfont]  {$\Delta t$};
		\draw (359.2,106) node [anchor=north west][inner sep=0.75pt]  [color={rgb, 255:red, 0; green, 0; blue, 255 }  ,opacity=1 ] [align=left] {\textit{{\footnotesize K}}};
		\draw (239.2,106) node [anchor=north west][inner sep=0.75pt]  [color={rgb, 255:red, 0; green, 0; blue, 255 }  ,opacity=1 ] [align=left] {\textit{{\footnotesize K}}};
		\draw (293.53,106) node [anchor=north west][inner sep=0.75pt]  [color={rgb, 255:red, 0; green, 0; blue, 255 }  ,opacity=1 ] [align=left] {\textit{{\footnotesize D}}};
		\draw (138.53,106) node [anchor=north west][inner sep=0.75pt]  [color={rgb, 255:red, 0; green, 0; blue, 255 }  ,opacity=1 ] [align=left] {\textit{{\footnotesize D}}};
		\draw (120.2,106) node [anchor=north west][inner sep=0.75pt]  [color={rgb, 255:red, 0; green, 0; blue, 255 }  ,opacity=1 ] [align=left] {\textit{{\footnotesize K}}};
	\end{tikzpicture}
	\caption{Illustration of KDMC. In the first time step of size $\dt$, the particle is in a high-collision regime such that the kinetic (K) step only lasts a short time and the diffusive (D) step dominates.  In the third time step of size $\dt$, the particle is in a low-collisional regime where only the K step is performed.}
	\label{fig: KDMC}
\end{figure}
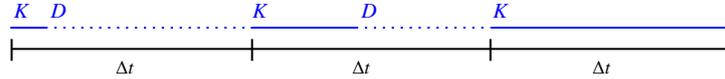

\subsection{Fluid estimation procedure for KDMC}\label{sec: fluid estimaion}

After the simulation, we calculate/estimate the moments \eqref{eqn: moments} based on the trajectories generated during the simulation \eqref{eqn: trajectory}. 
The contribution of the kinetic steps from $(x_i^k, v_i^k)$ to $({x_i^k}', {v_i^k}')$
can be estimated by MC estimators \cite{luxMonteCarloParticle2018a}. In our case,
the moments \eqref{eqn: moments} can be estimated as
\begin{align}\label{eqn: K moments}
\begin{split}
    m_l(x) =\int_0^{\bar{t}} m_l(x,t) dt =  \intv s_l(x, v) f(x,v) dv \approx \frac{1}{I}\sum_{i=1}^I \sum_{k=0}^{K} s_l\left(x_i^k, v_i^k, x_i^{k'}, v_i^{k'}\right) \delta(x-x_i^k). 
\end{split}
\end{align}
for $l=0, 1,$ and $2$, where $s_l(x,v)$ are score functions \cite{mortierAdvancedMonteCarlo2020} depending on the estimator, in other words, the estimating method. In this work, the track-length estimator \cite{mortierAdvancedMonteCarlo2020} is applied.
In fusion applications, discretization of the spatial domain is required \cite{reiterEIRENEB2EIRENECodes2005a}. The moments for the $j$-th cell $\mathcal{C}_j=[x_j, x_{j+1})$ are then 
\begin{equation}\label{eqn: K estimator}
    m_{l,j}=\int_{\mathcal{C}_j} m_l(x) dx \approx \frac{1}{I}\sum_{i=1}^I \sum_{k=0}^{K} s_l\left(x_i^k, v_i^k, x_i^{k'}, v_i^{k'}\right)  \mathbbm{1}_{\mathcal{C}_j}({x_i^k}),
\end{equation}
where $\mathbbm{1}_{\mathcal{C}_j}(x)$ is the indicator function of the $j$-th cell.
The process for obtaining the kinetic part of \eqref{eqn: K distribution} is called the kinetic simulation, and the estimation method given by \eqref{eqn: K estimator} is referred to as the MC estimator or the kinetic estimator.

For the diffusive simulation, in which the aggregate effect of many collisions is approximated with a single diffusive step, we only have incomplete trajectory information, and hence the MC estimator cannot produce a good estimate. 
To be specific, the position and velocity of each collision are required to produce unbiased estimates with the score function $s_l(x,v)$ in the MC estimator \eqref{eqn: K estimator}, 
while only the starting and end positions are provided in a diffusive step \eqref{eqn: SDE update}. 
To address this issue, we make use of the fluid estimation procedure proposed in \cite{mortierEstimationPostprocessingStep2022a}.

Since each diffusive step is designed to follow the particle-level dynamics of the fluid model \eqref{eqn: fluid-model} as discussed in Section \ref{sec: diffusion simulation}, we then assume that the ensemble of all diffusive steps collectively approximates the macroscopic fluid governed by \eqref{eqn: fluid-model}. 
The starting position of all diffusive steps gives the initial condition for the fluid motion. 
More precisely, the initial condition is
\begin{equation}\label{eqn: initial fluid estimator}
    \hat\rho_0(x) = \sum_{i=1}^I \sum_{k'=0}^{K'_i} w_i \delta(x-x_i^{k'}),
\end{equation}
where $K'_i$ is the number of diffusive steps of the $i$-th particle. 
Then, the moments \eqref{eqn: moments} are calculated by computing \eqref{eqn: m0}-\eqref{eqn: m2} after solving the fluid model \eqref{eqn: fluid-model} with the initial condition \eqref{eqn: initial fluid estimator}. 

A challenge in the fluid estimation is that the durations of all diffusive steps over all particles are different. 
As a result, the evolution time of the fluid model, denoted as $\theta$, needs to be chosen in a sensible way. 
Here, we take $\theta$ as the average duration of all diffusive steps.
That is, if the flight time of $i$-th particle in the $k$-th diffusive step is $\theta_i^{k}$, the empirical evolution time
\begin{equation}\label{eqn: fluid estimation time}
    \hat{\theta} = \sum_{i=1}^I \sum_{k'=0}^{K'_i} \frac{1}{I}\frac{1}{K_i'} \theta_i^{k}.
\end{equation}
The error in the fluid estimation, along with the choice of the empirical evaluation time, is analyzed in Section~\ref{sec: num. error of fluid model}. 
The complete algorithm of KDMC with the combined kinetic and fluid estimation described in this section is summarized in Algorithm \ref{alg: KDMC with f},
where the kinetic part and the fluid part moments are denoted as $m^{\text{kinetic}}$ and $m^{\text{fluid}}$, respectively.

\begin{algorithm}[h]
\caption{KDMC simulation with the associated fluid estimation procedure}
\label{alg: KDMC with f}
\begin{algorithmic}[1]
\renewcommand{\algorithmicrequire}{\textbf{Input:}}
\renewcommand{\algorithmicensure}{\textbf{Output:}}
\Require \parbox[t]{\linewidth}{Simulation time $\bar{t}$, time step $\Delta t$, number of particles $I$,\\
initial states $(x_i^0, v_i^0)$ for $i=1,\ldots,I$.}
\Ensure Moments $m$

\State $K \gets \bar{t} / \Delta t$ \Comment{Number of time steps}
\For{$i = 1$ to $I$}
    \For{$k = 0$ to $K-1$}
        \State $\tau \sim \text{Exp}\left(R(x_i^k)\right)$
        \State $\tau \gets \min(\tau, \Delta t)$ \Comment{Time of kinetic step}
        \State $x_i^{k'} \gets x_i^k + \tau v_i^k$ \Comment{Kinetic step (Eq. \eqref{eqn: a kinetic step})}

        \State $\theta_i^k \gets \max(\Delta t - \tau, 0)$ \Comment{Time of diffusive step}
        \If{$\tau < \Delta t$} \Comment{If diffusion step exists}
            \State Sample $v_i^{k'} \sim M\left(v | x_i^{k'}\right)$ and $v_i^{k+1}\gets v_i^{k'}$  \Comment{New velocity after collision}
            \State Sample $\xi \sim \mathcal{N}(0, 1)$
            \State $x_i^{k+1} \gets x_i^{k'} + A(x_i^{k'}) \theta_i^k + \sqrt{2 D(x_i^{k'}) \theta_i^k} \xi$ \Comment{Diffusion step (Eq.\eqref{eqn: SDE update})}
            \State Store $x_i^{k'}$ for the fluid estimation 
        \Else \Comment{If no diffusive step}
            \State $x_i^{k+1} \gets x_i^{k'}$, $v_i^{k'} \gets v_i^k$, and $v_i^{k+1} \gets v_i^k$
        \EndIf
        \State \label{line: mc_estimator} $m^{\text{kinetic}} \gets m^{\text{kinetic}} + \text{estimate}\left((x_i^k, v_i^k),(x_i^{k'}, v_i^{k'})\right)$ \Comment{Kinetic part moments with MC estimator}
        \State Store $\theta_i^k$ for the fluid estimation
    \EndFor
\EndFor

\State Construct the initial condition $\hat\rho_0(x)$ using all $x_i^{k'}$ as Eq. \eqref{eqn: initial fluid estimator}. \Comment{Fluid estimation}

\State Compute the empirical evolution time $\hat{\theta}$ using all $\theta_i^k$ as Eq. \eqref{eqn: fluid estimation time}.

\State \parbox[t]{\linewidth}{Solve the fluid model \eqref{eqn: fluid-model} with the initial condition $\hat\rho_0(x)$ up to the time $\hat{\theta}$, \\
and estimate fluid part moments $m^{\text{fluid}}$ as \eqref{eqn: m0}-\eqref{eqn: m2}.}
\State The total moments are given by $m= m^{\text{kinetic}} + m^{\text{fluid}}$.

\end{algorithmic}
\end{algorithm}

\section{Error of Kinetic-Diffusion Monte Carlo Simulation}\label{sec: KDMC analysis}
In this section, we derive an error bound for the KDMC simulation measured in the 1-Wasserstein ($W_1$) distance. 
The standard MC method introduced in Section \ref{sec: kinetic simulation} is used as a reference solution. 
In Section \ref{sec: WS}, we introduce two definitions of the $W_1$ distance in the one-dimensional case and explain its application by analyzing the movement of particles in the KDMC simulation. 
In Section \ref{sec: KDMC convergence}, we derive a bound for the error of the KDMC simulation, focusing on two limit cases: the diffusive scaling parameter $\varepsilon\rightarrow0$ and the time step $\dt\rightarrow0$. 
The KDMC simulation error is first derived at a fixed terminal time corresponding to the terminal-time simulation. Then, we demonstrate that the time-integrated simulation has a lower error.

\subsection{Wasserstein distance}\label{sec: WS}
To quantify the simulation error, we adopt the $W_1$ distance \cite{villaniOptimalTransportOld2009a} as our metric. Here, we introduce its two definitions for our one-dimensional case. 
\begin{definition}
Let $(\mathcal{X}, d)$ be a metric space. For any two probability measures $f, g$ on $\mathcal{X}$, the $W_1$ distance between $f$ and $g$ is defined as 
\begin{align}\label{Def: WS}
\begin{split}
W_1(f,g) 
&=\inf\left\{\EX[d(X,Y)], \mathrm{law}(X)=f, \mathrm{law}(Y)=g \right\}.
\end{split}
\end{align}
\end{definition}
The expectation $\EX$ of the metric $d$ is
$$\EX[d(X,Y)]=\iint d(x,y)f(dx)g(dy).$$
In our case, the law in \eqref{Def: WS} is the probability distribution of possible trajectories obtained from the kinetic or the KDMC simulation. 
Next, we explain how the particles' movement affects the $W_1$ distance.

Suppose the two empirical measures or simply distributions we compare are $f(x)$ with its samples $\{x_1, x_2, \cdots, x_I\}$ and  $g(y)$ with its samples $\{y_1, y_2, \cdots, y_I\}$, 
where $f(x)$ is the reference distribution, and the samples are positions of particles.
In the $W_1$ distance, we need to find the minimal cost that moves particles sampled from $g(y)$ such that $g$ and $f$ are identical distributions.  
A direct upper bound for $W_1$ is found by moving $y_i$ to $x_i$ for all $i = 0, 1, \cdots, I$, and the distance is $\sum_id(x_i, y_i)$. Taking the metric $d(x,y)=||x-y||_1$ and substituting into \eqref{Def: WS}, we have 
\begin{equation}\label{eqn: WS inequality}
   W_1(f, g)= \inf\{\EX[||x - y||_1]\} \leq\lim_{I\rightarrow\infty}\frac{1}{I}\sum_{i=1}^I |x_i - y_i|\approx\frac{1}{I}\sum_{i=1}^I |x_i - y_i|. 
\end{equation}
The equality holds when $x_i$ and $y_i$ are sorted in ascending order in one-dimensional cases. The bound in \eqref{eqn: WS inequality} without sorting may overestimate the true $W_1$ distance.
In the error derivation of Section \ref{sec: KDMC local error term}, some properties of the distributions are used to find a sharper upper bound than \eqref{eqn: WS inequality}. 

Alternatively, the $W_1$ distance can be defined in terms of the cumulative distribution functions (CDF).
\begin{definition}
Suppose $f(x)$ and $g(x)$ are normalized probability distributions on $\mathbb{R}$, and $F(x)$ and $G(x)$ are their CDFs, respectively. The $W_1$ distance between $f(x)$ and $g(x)$ is 
\begin{align}\label{Def: WS-2}
\begin{split}
W_1(f,g)&= \int_{\mathbb{R}} |F(x)-G(x)|dx.
\end{split}
\end{align}
\end{definition}
\begin{remark}\label{rek: ws computing}
Both definitions apply to the so-called balanced problem, where the total mass of the two distributions \( f \) and \( g \) is equal. This condition holds in our application, as no particles disappear (i.e., no ionization) and no new particles are born during the simulation. Consequently, the transport between \( f \) and \( g \) is mass-conserving, which ensures the well-posedness of the two definitions. 
Moreover, the distribution governed by the Boltzmann--BGK equation~\eqref{eqn: BBGK} is not necessarily normalized. 
Therefore, the distributions \( f(x) \) and \( g(x) \) should first be normalized for computing the $W_1$ distance. 
For readability, we omit the explicit notation of this normalization in the manuscript.

\end{remark}


\subsection{Convergence of KDMC}\label{sec: KDMC convergence}
We now derive the general expression of the $W_1$ distance between the distributions generated from the kinetic and the KDMC simulations. 

We denote the reference distribution obtained from the kinetic simulation at time $t^{(k+1)}=(k+1)\Delta t$ as $\fe{k+1}$ and its approximation obtained from KDMC as $\fp{k+1}$.
If we define $\St$ as the operator of applying the kinetic simulation to a distribution with one time step $\dt$ and $\Zt$ as that of the KDMC simulation, we have
$$\fe{k+1}=\St(\fe{k}) \text{ and } \fp{k+1}=\Zt(\fp{k}).$$
Note that only one diffusive step \eqref{eqn: SDE update} occurs in KDMC, whereas multiple kinetic steps \eqref{eqn: a kinetic step} could happen in the single time step $\dt$ in the kinetic simulation.
The $W_1$ distance between these two simulations at time $t^{(k+1)}$ is then 
\begin{align}
\nonumber
    W_1\!\left(\fe{k+1}, \fp{k+1}\right) &= W_1\!\left(\St(\fe{k}), \Zt(\fp{k})\right) \\ \label{eqn: 3_2_1}
    &\leq W_1\!\left(\St(\fe{k}), \St(\fp{k})\right) + W_1\!\left(\St(\fp{k}), \Zt(\fp{k})\right),  
\end{align}
where the triangle inequality of $W_1$ distance \cite{villaniOptimalTransportOld2009a} is applied. 
The first term in \eqref{eqn: 3_2_1} is the difference of applying the reference scheme, that is, the kinetic simulation, to $\fe{k}$ and $\fp{k}$, resulting in the different initial conditions to the time step. 
We show in Section \ref{sec: KDMC global error term} that this term is bounded by the global error at the previous time step $t^{(k)}$.
The second term in \eqref{eqn: 3_2_1} captures the error of performing one time step of the kinetic and the KDMC simulation on the same distribution $\fp{k}$, and thus corresponds to the local error of KDMC compared to the kinetic simulation. 
This term is studied in Section \ref{sec: KDMC local error term}, particularly in the limits: 
$\varepsilon\rightarrow0$, and $\dt\rightarrow0$. 
Finally, we provide a bound for the $W_1$ distance between the kinetic and the KDMC simulations in Section \ref{sec: sub kdmc error}.



\subsubsection{Global error term}\label{sec: KDMC global error term}
This section shows that the first term in \eqref{eqn: 3_2_1}, i.e., 
\begin{equation} \label{eqn: ws global}
    W_1\!\left(\St(\fe{k}), \St(\fp{k})\right),    
\end{equation}
is bounded by $W_1\!\left(\fe{k}, \fp{k}\right)$, which is the global error from the previous step. As a result, this term represents the propagation of the global error.

The expression \eqref{eqn: ws global} indicates the $W_1$ distance between distributions produced 
by the kinetic simulation with the initial distributions $\fe{k}$ and $\fp{k}$, respectively, up to time $\dt$. 
The kinetic simulation is unbiased, which means the resulting distribution represents the true solution of the Boltzmann-BGK equation \eqref{eqn: BBGK} plus a statistical error that becomes negligible if a large enough number of particles is used. 
We first assume that an infinite number of particles is used so that the statistical error vanishes, and leave the discussion of statistical error to Section \ref{sec: sub kdmc error}. 
With this assumption, the operator $\St(f)$ is equivalent to evolving the initial distribution $f$ with time $\dt$ governed by the Boltzmann-BGK equation \eqref{eqn: BBGK}. 
Then, the following theorem shows that the operator $\St$ is contractive under the $W_1$ metric, thereby providing the error bound for \eqref{eqn: ws global}.
\begin{theorem}\label{thm: ws global error bound}
Let $f_0^1, f_0^2$ be two initial distributions of particles. With an infinite number of particles, $f^1_{\dt}=\St(f_0^1)$ and $f^2_{\dt}=\St(f_0^2)$ are corresponding solutions of the Boltzmann-BGK equation \eqref{eqn: BBGK} at the time $\dt$. Then
\begin{equation} \label{eqn: WS Lip}
    W_1\!\left(f^1_{\dt}, f^2_{\dt}\right) \leq W_1\!\left(f^1_0, f^2_0\right),
\end{equation}
\end{theorem}

\begin{proof}
To prove this inequality, we first express the solution of the Boltzmann-BGK equation \eqref{eqn: BBGK} as the so-called mild formulation \cite{dipernaCauchyProblemBoltzmann1989} given by Duhamel's principle,
then apply the $W_1$ distance on $f^1_{\dt}$ and $f^2_{\dt}$. 
Afterwards, this inequality is obtained using Grönwall's inequality. 
In this proof, we assume a homogeneous background $R(x)=R$. 
The argument can, however, be performed in a heterogeneous background, for which the corresponding mild formulation can be found in \ref{appendix: mild formulation}. 
Distributions, such as $f_0^1$, depend on $(x,v)$, and the dependency is written explicitly if necessary.

The mild formulation of the Boltzmann-BGK equation \eqref{eqn: BBGK} (see \ref{appendix: mild formulation} for the derivation) reads 
\begin{equation}
    f(x,v,t) = e^{-Rt}f_0(x-vt, v) + \int_0^t Re^{-R(t-s)}M_p[f](x-v(t-s),v, s)ds,
\end{equation}
with 
\begin{equation*}
    M_p[f](x,v,t) = M_p(v|x)\intv f(x,v',t)dv',
\end{equation*}
and the initial condition $f(x,v,t=0)=f_0(x,v)$. 
Then,
\begin{equation}
    f^i_{\dt} = e^{-R\dt}f_0^i(x-v\dt, v) + \int_0^{\dt} Re^{-R(\dt-s)}M_p[f^i](x-v(\dt-s),v, s) ds.
\end{equation} 
with $i=1,2$.
In addition, we have
\begin{equation*}
    W_1\!\left(M_p[f^1], M_p[f^2]\right)\leq W_1\!\left(f^1, f^2\right),
\end{equation*}
since $M[f]$ projects the full phase-space density $f(x,v,t)$ onto a Maxwellian (in $v$) that depends only on the local density $\int f(x,v',t)dv'$ at time $t$. 
In other words, the projection $f\mapsto M[f]$ smooths the distribution by removing the variation in $v$, and a lower cost is required to transport the distribution after the projection. 

Next, applying the $W_1$ distance on $f^1_{\dt}$ and $f^2_{\dt}$, we have
\begin{equation}\label{eqn: gronwall assumption}
    W_1\!\left(f^1_{\dt}, f^2_{\dt}\right) \leq e^{-R\dt } W_1\!\left(f^1_0, f^2_0\right) + \int_0^{\dt} Re^{-R(\dt-s)}W_1\!\left(f^1_{s}, f^2_{s}\right) ds,
\end{equation}
with the triangle inequality. We also use the fact that
$W_1\!\left(f^1_0(x-v\dt,v), f^2_0(x-v\dt,v)\right) = W_1\!\left(f^1_0(x,v), f^2_0(x,v)\right)$ since the $W_1$ metric is isometry invariant \cite{panaretosStatisticalAspectsWasserstein2019}.

Finally, multiplying $e^{R\dt }$ on both sides of \eqref{eqn: gronwall assumption} and by Grönwall's inequality \cite{gronwallNoteDerivativesRespect1919}, the inequality \eqref{eqn: gronwall assumption} becomes
\begin{equation*}
    e^{R\dt } W_1\!\left(f^1_{\dt}, f^2_{\dt}\right)  \leq W_1\!\left(f^1_0, f^2_0\right) \exp\left(\int_0^{\dt}Re^{-R(\dt-s)}ds\right).
\end{equation*}
Further, we obtain
\begin{equation*}
    W_1\!\left(f^1_{\dt}, f^2_{\dt}\right)  \leq W_1\!\left(f^1_0, f^2_0\right)\exp\left(1-e^{-R\dt}-R\dt\right) \leq W_1\!\left(f^1_0, f^2_0\right). \qedhere
\end{equation*}
\end{proof}
Substituting $f^1_0 = \fe{k}$, $f^2_0=\fp{k}$, $f^1_{\dt} = \St(\fe{k})$, and $f^2_{\dt}=\St(\fp{k})$ into \eqref{eqn: WS Lip}, we have 
\begin{equation} \label{eqn: ws globl error term}
    W_1\!\left(\St(\fe{k}), \St(\fp{k})\right)    \leq W_1\!\left(\fe{k}, \fp{k}\right).
\end{equation}
Substituting \eqref{eqn: ws globl error term} into \eqref{eqn: 3_2_1}, we obtain
\begin{equation} \label{eqn: total simulation formula}
    W_1\!\left(\fe{k+1}, \fp{k+1}\right)  
    \leq W_1\!\left(\fe{k}, \fp{k}\right)  + W_1\!\left(\St(\fp{k}), \Zt(\fp{k}) \right).
\end{equation}
This inequality states that the global error at the $(k+1)$-th time step is bounded by the error from the previous time step plus the local error introduced during the current step.

\subsubsection{Local error term}\label{sec: KDMC local error term}
In this section, we discuss the local error term 
\begin{equation} \label{eqn: local err}
    W_1\!\left(\St(\fp{k}), \Zt(\fp{k})\right),
\end{equation}
in \eqref{eqn: 3_2_1}. 
The subsequent two theorems state that the local error term \eqref{eqn: local err} is bounded by $\BO(\varepsilon^2)$ as the diffusive scaling parameter $\varepsilon\rightarrow 0 $ 
and $\BO(\dt^2)$ as the time step $\dt\rightarrow 0 $, respectively.

\begin{theorem}
Let $\fp{k}$ be the distribution obtained from the KDMC simulation at time $t^{(k)}=k\Delta t$. 
Define $\St$ as the operator of applying the kinetic simulation to a distribution with one time step $\dt$ and $\Zt$ as that of the KDMC simulation.
As $\varepsilon\rightarrow 0 $,      
\begin{equation}\label{eqn: local error e}
    W_1\!\left(\St(\fp{k}), \Zt(\fp{k})\right) \leq  \BO(\varepsilon^2).
\end{equation}
\end{theorem}

\begin{proof}
Denote $f_k$ as the solution of \eqref{eqn: BBGK} simulated by the kinetic simulation described in Section \ref{sec: kinetic simulation}, and $f_f$ as the solution of \eqref{eqn: fluid-model} simulated by the biased SDE \eqref{eqn: SDE update}, both evaluated at time $t^{(k)}$. To show that the local error term \eqref{eqn: local err} is bounded by $\BO(\varepsilon^2)$ as $\varepsilon\rightarrow0$,  we first find the $W_1$ distance between $f_k$ and $f_f$, and then bound \eqref{eqn: local err} by this distance.

As mentioned in Section \ref{sec: fluid model}, the approximation error between the Boltzmann-BGK \eqref{eqn: BBGK} and the fluid model \eqref{eqn: fluid-model} is $\BO(\varepsilon^2)$. 
Since $f_k$ is the numerical solution of \eqref{eqn: BBGK} with only the statistical error and $f_f$ is the numerical solution of \eqref{eqn: SDE update} with only the discretization error,  
we have $|f_k - f_f|=\BO(\varepsilon^2)$, if the statistical error and the discretization are assumed to be negligible. Assume the distribution be on the domain $\Omega=[a, b]$, then
\begin{equation*}
    F_k(x) - F_f(x) = \int_a^x f_k(y) - f_f(y) dy = \BO(\varepsilon^2),
\end{equation*}
where functions $F_k$ and $F_f$ are the CDFs of $f_k$ and $f_f$, respectively.
Further, the $W_1$ distance between $f_k$ and $f_f$ is 
\begin{equation}\label{eqn: WS app 1}
    W_1(f_k, f_f) = \int_a^b |F_k - F_f| dx = \int_a^b \left|\int_a^x f_k(y) - f_f(y) dy\right| dx = \BO(\varepsilon^2),
\end{equation}
where the second definition of the $W_1$ distance \eqref{Def: WS-2} is used. 

If $f_k$ is simulated with the initial distribution $\fp{k}$, then $f_k=\St(\fp{k})$. 
As for $\Zt(\fp{k})$, it is obtained by a combination of the biased SDE \eqref{eqn: SDE update} (diffusive step) and the unbiased kinetic simulation in Section \ref{sec: kinetic simulation} (kinetic step). 
Thus, the distance between $\Zt(\fp{k})$ and $\St(\fp{k})$ is less than the distance between $f_k=\St(\fp{k})$ and $f_f$, since $f_f$ is simulated purely by the biased SDE. Consequently, the local error 
\begin{equation*}
    W_1\!\left(\St(\fp{k}), \Zt(\fp{k})\right) \leq W_1(\St(\fp{k}), f_f) = W_1(f_k, f_f) = \BO(\varepsilon^2). \qedhere
\end{equation*}
\end{proof}

The following theorem provides a bound for the local error term \eqref{sec: KDMC local error term} with respect to $\dt$ as $\dt\rightarrow0$.
\begin{theorem}
Let $\fe{k}$ be the distribution obtained from the kinetic simulation at time $t^{(k)}=k\Delta t$. 
Define $\St$ as the operator of applying the kinetic simulation to a distribution with one time step $\dt$ and $\Zt$ as that of the KDMC simulation.
As $\dt\rightarrow0$,
\begin{equation}\label{eqn: local error dt}
W_1\!\left(\St(\fp{k}), \Zt(\fp{k})\right) \leq \BO(\frac{\dt^2}{\varepsilon^2}).
\end{equation}
\end{theorem}
\begin{proof}
We use the inequality \eqref{eqn: WS inequality} from the first definition of the $W_1$ distance, which involves the position of particles. 
Consider two particles that start from the same initial position but evolve under the kinetic and the KDMC simulations, respectively.  
Let  $x^k$  denote the terminal position of the particles in the kinetic simulation and $x^{kd}$ denote that in the KDMC simulation. 
Within one time step of size $\dt$, two situations are possible: 
\begin{itemize}
    \item \textbf{Neither particles collide in the time step.} In this case, both simulations have the same initial states and the same flight time $\dt$, so the final positions are the same. Thus, no difference between them, so
$$W_1\!\left(\St(\fp{k}), \Zt(\fp{k})\right)=0.$$
The probability of a single particle not colliding in $\dt$ is $e^{-R\dt}=1-R\dt+\BO(\dt^2)$ when $\dt\rightarrow0$.
Then the probability for this case is $e^{-R\Delta t}e^{-R\Delta t}=1-\BO(R\dt)=1-\BO(\dt/\varepsilon^2)$. 

    \item \textbf{At least one particle collides in the time step.} In this case, the distance $|x^k - x^{kd}|$ is bounded by $\BO(\dt)$ since the flight time of both particles is less than or equal to $\dt$. 
    The boundedness holds for all particles following the initial distribution $\fp{k}$. Thus, applying the kinetic and the KDMC simulation on $\fp{k}$, we have
\begin{equation*}
W_1\!\left(\St(\fp{k}), \Zt(\fp{k})\right) \leq \BO(\dt).
\end{equation*}
The probability for this case is $1-e^{-R\Delta t}e^{-R\Delta t} = \BO(\dt/\varepsilon^2)$, 
\end{itemize}
The weighted sum of these two cases gives the local error  
\begin{equation*}
W_1\!\left(\St(\fp{k}), \Zt(\fp{k})\right) \leq \BO(\frac{\dt^2}{\varepsilon^2}). \qedhere
\end{equation*}
\end{proof}

\subsubsection{KDMC simulation error}\label{sec: sub kdmc error}
Before concluding the analysis of the KDMC simulation error, we briefly comment on the statistical error, which was omitted in the proof of the global error in Section~\ref{sec: KDMC global error term}.
When simulating in practice, a finite number of samples is used, leading to the statistical error of $\sigma/\sqrt{N}$ \cite{lapeyreIntroductionMonteCarlo2003a}, 
where $\sigma$ is the standard deviation of the underlying distribution, and $N$ is the number of samples.
This section (Section \ref{sec: KDMC analysis}), which concerns the simulation part,  mainly focuses on the deterministic errors (also called the bias), e.g., \eqref{eqn: local error e} and \eqref{eqn: local error dt}, and the statistical error is assumed to be low in both the KDMC simulation and the reference kinetic simulation, since we can always use more particles to reduce this error. For simplicity, this error is denoted as $\eta_s=\BO(1/\sqrt{N})$. 
We refer to \cite{ghoosAccuracyConvergenceCoupled2016, baetenAnalyticalStudyStatistical2018a}  for the detailed discussion of the statistical error in MC simulation. 

Next, the following theorem presents the main error bound for the KDMC simulation.

\begin{theorem}
Let $\fe{k+1}$ be the reference distribution obtained from the kinetic simulation at time $t^{(k+1)}=(k+1)\Delta t$ and its approximation obtained from KDMC be $\fp{k+1}$. 
The error of the KDMC simulation compared with the kinetic simulation, 
which is the $W_1$ distance between $\fe{k+1}$ and $\fp{k+1}$, written as $\epsilon_s$, reads
\begin{equation}\label{eqn: kdmc error}
 \epsilon_{s} = W_1(\fe{k+1}, \fp{k+1}) \leq \eta_s +
\begin{cases} 
      \BO\left(\frac{\varepsilon^2}{\dt}\right),  & \text{if } \dt \gg \varepsilon^2, \\
      \BO\left(\frac{\dt}{\varepsilon^2}\right),  & \text{if } \dt \ll \varepsilon^2,
\end{cases}
\end{equation}
where $\eta_s=\BO(1/\sqrt{N})$ is the statistical error.
\end{theorem}

\begin{proof}   
Recursively substituting the local error \eqref{eqn: local error e} into \eqref{eqn: total simulation formula}, the $W_1$ distance between the two simulations at time $t^{(n+1)}$ is 
\begin{align}\label{eqn: global error e}
\begin{split}
    W_1(\fe{k+1}, \fp{k+1})  & \leq
    W_1\left(\fe{k}, \fp{k}\right) + \BO(\varepsilon^2) \leq \BO(\frac{\varepsilon^2}{\dt}t^{(k+1)}) = \BO(\frac{\varepsilon^2}{\dt}),
\end{split}
\end{align}
as $\varepsilon\rightarrow 0$, given that the initial error $W_1(\fe{0}, \fp{0})$ is zero. 
Therefore, by fixing the time step $\dt$, the global error \eqref{eqn: global error e} of the KDMC simulation is $\BO(\varepsilon^2)$.
Similarly, substituting the local error \eqref{eqn: local error dt} into \eqref{eqn: total simulation formula}, as $\dt\rightarrow0$,
\begin{align}\label{eqn: global error dt}
\begin{split}
    W_1(\fe{k+1}, \fp{k+1})  \leq \BO(\frac{\dt}{\varepsilon^2}), 
\end{split}
\end{align}
which means the global error \eqref{eqn: global error dt} is $\BO(\dt)$. 
Finally, the combination of \eqref{eqn: global error e} and \eqref{eqn: global error dt} leads to \eqref{eqn: kdmc error}, where the two limits $\epsilon\rightarrow0$ and $\dt\rightarrow0$ are replaced by $\dt \gg \varepsilon^2$ and $\dt \ll \varepsilon^2$, respectively, under the assumption that $\dt$ and $\varepsilon$ are both small, for emphasizing the relation between $\varepsilon$ and $\dt$.
The statistical error $\eta$ in \eqref{eqn: kdmc error} is due to the finite number of samples in the simulation, and its linear contribution to the total error is justified by the linearity of \eqref{eqn: BBGK}.
\end{proof}

The error \eqref{eqn: kdmc error} is between the distributions simulated from the kinetic simulation and the KDMC simulation at a time $t=(k+1)\dt$, which corresponds to the time-terminal simulation in \ref{appendix: steady-state solution}. If the time-integrated simulation is applied, which accumulates the distribution at all times $t_l = l\dt$ with $l=1,2,\ldots,k$, a smaller error is typically observed. This fact can be seen as follows.

The normalized time-integrated distributions of $\fe{k}$ and $\fp{k}$ at the time $t^{(k)}$ are
\begin{equation}
    \frac{\sum_{l=1}^k \fe{l}}{\int_\Omega \sum_{l=1}^k \fe{l} dx} = \frac{\sum_{l=1}^k \fe{l}}{k \int_\Omega \fe{0} dx}, \quad \text{and} \quad \frac{\sum_{l=1}^k \fp{l}}{\int_\Omega \sum_{l=1}^k \fe{l} dx} = \frac{\sum_{l=1}^k \fp{l}}{k \int_\Omega \fp{0} dx},
\end{equation}
where $\Omega$ is the spatial domain of the distributions, and the equalities are due to the mass-conservation as explained in Section \ref{sec: WS}. We also have the equality of the initial mass
$\int_\Omega \fe{0} dx = \int_\Omega \fp{0} dx$ since the initial condition is the same. 
Then, the $W_1$ distance between the two time-integrated simulations is  
\begin{equation}
    W_1\!\left(\frac{1}{k}\sum_{l=1}^k\fe{l}, \frac{1}{k}\sum_{l=1}^k\fp{l}\right)\leq \frac{1}{k}\sum_{l=1}^k   W_1\!\left(\fe{l}, \fp{l}\right) \leq  W_1\!\left(\fe{k}, \fp{k}\right),
\end{equation}
where the first inequality is due to the convexity of the $W_1$ metric \cite{panaretosStatisticalAspectsWasserstein2019}, and the last quantity is $\epsilon_s$ in \eqref{eqn: kdmc error}. The normalization factors $\int_\Omega \fe{0} dx $ and $ \int_\Omega \fp{0} dx$ are not written explicitly as stated in Remark \ref{rek: ws computing}.


\section{Estimation error}\label{sec: estimation sec}
KDMC is a hybrid approach that consists of kinetic and diffusive simulations.
In this work, the standard MC estimator \eqref{eqn: K estimator} is used to score, i.e., to compute, the moments of the kinetic part, and the fluid estimation, derived in Section \ref{sec: fluid estimaion}, is used to score the moments of the diffusive part. 
The overall estimation error is the combination of errors in these two estimation parts. 
To analyze the estimation error, we decompose it into the kinetic and diffusive parts in Section \ref{sec: ana estimation error}, which are then analyzed separately in Sections~\ref{sec: error of kinetic part} and~\ref{sec: num. error of fluid model}, respectively.

The reference solution is computed by the standard MC method, which simulates particle trajectories as in Section \ref{sec: kinetic simulation} and estimates the moment using the standard MC estimator \eqref{eqn: K estimator}. 
For brevity, we refer to this method as the kinetic method. The KDMC simulation introduced in Section \ref{sec: KDMC} together with the associated fluid estimation described in Section \ref{sec: fluid estimaion} is referred to as the KDMC method or simply KDMC.
The moments \eqref{eqn: moments} are smooth functions, so we choose the relative $L_2$ norm to measure errors.

\subsection{Estimation error decomposition}\label{sec: ana estimation error}
We denote the reference moment obtained from the kinetic method by $m^{k}$, and the approximate moment obtained from KDMC by $m^{kd}$,
which includes the kinetic component $m^{kd,k}$  and the diffusive component $m^{kd,d}$, that is, $m^{kd} =  m^{kd, k} +  m^{kd, d}$.
Let $\alpha\in[0, 1]$ denote the expected fraction of the duration of the kinetic step within a time step of size $\dt$. In a homogeneous background, where the collision rate \( R(x) = R \) is constant, 
the coefficient $\alpha$ remains constant within each time step \( \Delta t \) and thus throughout the entire simulation. 
Under the $L_2$ norm, the relative estimation error is
\begin{align}\label{eqn: e_kd}
    \epsilon_{kd} =  \left\Vert m^{k} - m^{kd}\right\Vert_2/{\left\Vert m^{k}\right\Vert_2}  =  \left\Vert \alpha m^{k}-m^{kd,k} + (1-\alpha)m^{k}-m^{kd,d}\right\Vert_2/{\left\Vert m^{k}\right\Vert_2} \leq \alpha\epsilon_k + (1-\alpha)\epsilon_{kd},
\end{align}
where $m^{kd} =  m^{kd, k} +  m^{kd, d}$ is substituted in, and
\begin{equation}\label{eqn: e_k e_d def}
    \epsilon_{k} = \left\Vert \alpha m^{k} - m^{kd,k}\right\Vert_2/\left(\alpha\left\Vert m^{k}\right\Vert_2\right), \quad \text{and} \quad \epsilon_d = \left\Vert (1-\alpha)m^{k} - m^{kd,d}\right\Vert_2/\left((1-\alpha)\left\Vert m^{k}\right\Vert_2\right), 
\end{equation}
are defined as the relative kinetic and diffusive estimation errors, and discussed in Sections \ref{sec: error of kinetic part} and \ref{sec: num. error of fluid model}, respectively.  
Intuitively, the moment from the kinetic part $m^{kd,k}$ contributes an $\alpha$ portion of the approximate moment $m^{kd}$ 
and its error should therefore be measured relative to the corresponding portion of the reference moment, $\alpha m^{k}$. 
The same interpretation holds for the diffusive part and its corresponding relative error $\epsilon_d$.

To derive the coefficient $\alpha$, let $\tau\sim \text{Exp}(R)$ be the kinetic flight time of a particle in a time step $\dt$. The particle flies for time $\tau$ if $\tau<\dt$; otherwise, it flies for $\dt$. So the expectation is the weighted sum of these two events, i.e.,
\begin{equation*}
    \EX[\tau] = \int_0^{\dt} \tau R e^{-R\tau}d\tau + \dt e^{-R\dt} =\left(1-e^{-R\dt}\right)/{R} ,
\end{equation*}
where $e^{-R\dt}$ is the probability that the particle does not collide in $\dt$. Then the coefficient
\begin{equation}\label{eqn: alpha}
    \alpha = {\EX[\tau]}/{\dt} = \left({1-e^{-R\dt}}\right)/{R\dt}.
\end{equation}
Since $R\dt\in(0,\infty)$, we have limits 
\begin{equation*}
     \lim_{R\dt\rightarrow0} \alpha = 1, \quad \text{and} \quad \lim_{R\dt\rightarrow\infty} \alpha =  0.
\end{equation*}
From $R\dt = 0$ to $\infty$, the coefficient $\alpha$ monotonically decreases from $1$ to $0$. In practice, we assume  $R=\BO(1/\varepsilon^2)\gg1$ and $\dt\ll1$.
Then, we can conclude from \eqref{eqn: e_kd} that 
\begin{itemize}
    \item If $\dt\gg\varepsilon^2$, that is $\varepsilon\rightarrow0$, then $\alpha\rightarrow0$ and the diffusive part $m^{kd,d}$ dominates the estimated moment $m^{kd}$.
    \item If $\dt\ll \varepsilon^2$, that is $\dt\rightarrow0$,  then $\alpha\rightarrow1$ and the kinetic part $m^{kd,k}$ dominates the estimated moment $m^{kd}$.
\end{itemize} 
Hence, the estimation error \eqref{eqn: e_kd} can be approximated as
\begin{equation}\label{eqn: case e_kd}
    \epsilon_{kd} \approx 
    \begin{cases} 
      \epsilon_d,  & \text{if } \dt \gg \varepsilon^2, \\
      \epsilon_k, & \text{if } \dt \ll \varepsilon^2.
\end{cases}
\end{equation}
This formula \eqref{eqn: case e_kd} is based on the homogeneous background assumption, that is, we assume $R(x)=R$ is constant. 
In fact, for a moderate heterogeneous background, this formula still holds.
By moderate, we refer to cases where $R(x)=r(x)/\varepsilon^2$, and the range of the function $r(x)$ is bounded such that $R(x)=\BO(1/\varepsilon^2)$, which is also the assumption of the fluid limit in Section \ref{sec: fluid model}. 
Under the moderate assumption, the diffusive step remains dominant in each time step of size $\dt$, and thus $\epsilon_{kd}=\epsilon_d$ if $\dt\gg \varepsilon^2$; the reverse holds when  $\dt\ll \varepsilon^2$.
In Section \ref{sec: fusion test case}, we numerically show that \eqref{eqn: case e_kd} holds for the moderate heterogeneous background with a fusion-relevant test case.
Next, we discuss $\epsilon_k$ and $\epsilon_d$ separately, and subsequently combine them to derive the expression in \eqref{eqn: case e_kd}.
\subsection{Error of kinetic part $\epsilon_k$}\label{sec: error of kinetic part}
We discuss the kinetic error in two limiting cases: $\dt\rightarrow 0$ and $\varepsilon\rightarrow0$. 
In the kinetic part, the moments $m^{kd,k}$ defined in Section \ref{sec: ana estimation error} are scored by the unbiased MC estimator \eqref{eqn: K estimator}. Thus, the error is purely statistical.

\subsubsection{Error of kinetic part as $\dt\rightarrow 0$}
As explained in Section \ref{sec: sub kdmc error}, the statistical error is $\sigma/\sqrt{N}$ with the standard deviation $\sigma$ and the number of samples $N$.
In each time step $\dt$, one and only one kinetic flight happens.
If the total number of particles $I$ is fixed, more segments result in more samples for the estimator. Thus, the error is then lower.
In particular, the number of time steps is $K=\Bar{t}/\dt$ where $\Bar{t}$ is the simulation time, and the total number of samples $N$ used in the MC estimator is $N=KI$. Thus, the error scales as $\sigma/\sqrt{KI}=\sigma\sqrt{\dt}/\sqrt{I}=\BO(\sqrt{\dt})$ if $\sigma$ and $I$ are fixed. 

However, this error cannot vanish entirely. For instance, if a collision-based estimator is used, where only collisions contribute to the estimate. 
The probability that a particle does not collide in a time step $\dt$ is $e^{-R\dt}$ with $R=\BO(1/\varepsilon^2)$. Thus, the particle is not likely to collide in the time step, when $\dt\ll\varepsilon^2$.
Further reducing $\dt$ does not reduce this statistical error.
If a track-length estimator is used, which scores the flight distance instead of the collision, successive samples become highly correlated when $\dt\ll\varepsilon^2$. The effective number of independent samples no longer scales with the number of time steps. 
Therefore, the statistical error stagnates. 
So, if denoting the stagnated error as $\eta_k$, we have 
\begin{equation}\label{eqn: k error dt}
 \epsilon_{k} = \eta_k, \quad \text{if } \dt \ll \varepsilon^2,
\end{equation}
where $\eta_k=\BO(1/\sqrt{I})$ since $\dt$ has no effect on the statistical error if $\dt \ll \varepsilon^2$, as explained above. The error of the kinetic part dominates only if $\dt \ll \varepsilon^2$ as shown in \eqref{eqn: case e_kd}, and thus the case $\epsilon_k=\BO(\sqrt{\dt})$ if $\dt \gg \varepsilon^2$ is not important.
 
\subsubsection{Error of kinetic part as $\varepsilon\rightarrow0$}\label{sec: kinetic error e}
We assume the time step $\dt$ and the number of particles $I$ are fixed, and start from the collision rate being small.
When $\varepsilon\rightarrow\infty$, the collision rate $R=\BO\left(1/\varepsilon^2\right)\rightarrow 0$, and no collisions happen. 
As a result, in both the kinetic and the KDMC simulations, 
particles fly with the initial velocity over the whole simulation time $\Bar{t}$, 
which causes no difference between the moments estimated by the two methods.

As $\varepsilon$ decreases, the number of collisions $R\Bar{t}=\BO(1/\varepsilon^2)$ increases, with $\Bar{t}$ the simulation time. 
These collisions lead to different trajectories for the two simulations, and this difference introduces statistical fluctuations, which contribute to the variance of the MC estimator.
The variance $\sigma^2$ is then proportional to the number of collisions, and thus to $\BO(1/\varepsilon^2)$. 
The statistical error $\sigma/\sqrt{N}$ is then  $\BO(1/(\varepsilon\sqrt{N}))=\BO(1/\varepsilon)$ as $N=KI$ is fixed.
When $\dt\gg\varepsilon^2$, particles are likely to collide in each $\dt$. 
The variance, and therefore the statistical error, stagnates. 
The error of the kinetic part dominates only when $\dt\ll\varepsilon^2$ as illustrated in \eqref{eqn: case e_kd}, and thus it is not need to discuss the error of the kinetic part when  $\dt\ll\varepsilon^2$.
Finally, the error of the kinetic part is small when $\varepsilon$ is large, then it increases as $\BO(1/\varepsilon)$ when $\dt\ll\varepsilon^2$. Thus, we have
\begin{equation}\label{eqn: k error e}
 \epsilon_{k} = \BO\left(\frac{1}{\varepsilon}\right),  \quad \text{ if } \dt \ll \varepsilon^2.
\end{equation}
We remark that the $\dt$ affects the statistical error by changing the number of samples, and the diffusive scaling parameter $\varepsilon$ affects the statistical error by changing the variance $\sigma^2$.
In both limit cases, the statistical error starts to stagnate around $\dt=\varepsilon^2$. Thus, the two stagnation errors are expected to be of comparable magnitude.  

In fact, the variance of the kinetic part is also affected by the mean plasma velocity $u_p$, especially for the $m_1$ \eqref{eqn: moments}. A detailed discussion of $u_p$ is beyond the scope of this work, but we mention it briefly here.
The MC estimator \eqref{eqn: K estimator}, like the track-length and collision estimator \cite{mortierAdvancedMonteCarlo2020}, involves the sampling of $v^l$ with the particle velocity $v$ following the normal distribution $ N(u_p, \sigma_p^2)$  and $l = 0,1,2$. 
When $l$ is odd, the value $v^l$ fluctuates around zero, especially when the mean $u_p$ is small.
In this case, the summation of negative and positive samples in the estimator \eqref{eqn: K estimator} may cancel each other out, causing a large variance.
However, this is not the case if $l$ is even, since $v^l$ would always be positive. 
A rigorous proof can be done by writing down the expression of the variance for the sample $v^l$. 
We leave the proof for future work and conclude that 
if the MC estimation \eqref{eqn: K estimator} is used, the $m_1$ moment has a greater variance than the $m_0$ and $m_2$ moments when $u_p$ is small. This fact is shown numerically in Section \ref{sec: est kinetic}.

\subsection{Error of diffusive part $\epsilon_d$}\label{sec: num. error of fluid model}
The diffusive part estimation consists of multiple error components. 
We first analyze each error and subsequently explain how these errors combine to form the overall error of the diffusive part.
\subsubsection{Propagated error from KDMC simulation $\epsilon_{i}$}
The fluid estimation scores moments by solving the fluid model \eqref{eqn: fluid-model} with the initial condition $\hat\rho_0(x)$ provided by KDMC as shown in \eqref{eqn: initial fluid estimator}. 
However, this initial condition is not the exact initial particle distribution since the KDMC simulation introduces an error $\epsilon_{s}$ as shown in \eqref{eqn: kdmc error}. 
This initial error propagates through the fluid model solution. Below, we show that the resulting propagated error, denoted as $\epsilon_i$, has the same asymptotic bound as $\epsilon_s$. 

Let the true initial distribution be $\rho_0(x)$ and the corresponding true solution of the fluid model be $\rho(x,t)$. 
Similarly, let the solution given by the inexact initial condition $\hat\rho_0(x)$ be $\hat\rho(x,t)$. 
Rewrite the fluid model \eqref{eqn: fluid-model} in the operator form 
\begin{equation}
    \partial_t\rho(x,t)=\mathcal{L}\rho(x,t),
\end{equation}
where $\mathcal{L}$ is the linear spatial operator in \eqref{eqn: fluid-model}. 
Since this model is linear, its solution can be expressed as 
\begin{equation}
    \rho(x,t) = e^{-\mathcal{L}t}\rho(x,t=0).
\end{equation}
The propagated error in the solution at time $t$ is then given by
\begin{align}\label{eqn: e_i bound}
    \epsilon_i &= \lVert\rho(x,t)-\hat\rho(x,t)\rVert_2 
    =\lVert e^{-\mathcal{L}t}\left(\rho_0(x) - \hat\rho(x)\right)\rVert_2  
    \leq e^{-ct}\lVert \rho_0(x) - \hat\rho_0(x)\rVert_2,
\end{align}
where $c$ is a constant. The term $\lVert \rho_0(x) - \hat\rho_0(x)\rVert_2$ is the difference of the initial conditions under the $L_2$ norm, 
which is the KDMC error $\varepsilon_s$ described in \eqref{eqn: kdmc error}. 
Since $e^{-ct}$ in \eqref{eqn: e_i bound} is independent of $\varepsilon$ and $\dt$, the propagated error $\epsilon_i$ has the same convergence as $\epsilon_s$ with respect to $\varepsilon$ and $\dt$.
Therefore, 
\begin{equation}\label{eqn: initial error}
 \epsilon_{i} = \eta_s +
\begin{cases} 
      \BO\left(\frac{\varepsilon^2}{\dt}\right),  & \text{if } \dt \gg \varepsilon^2, \\
      \BO\left(\frac{\dt}{\varepsilon^2}\right),  & \text{if } \dt \ll \varepsilon^2,
\end{cases}
\end{equation} 
where $\eta_s$ is a statistical error that we assume to be negligible.        
\begin{remark}[Errors with different norms]\label{rek: mixed norms}
The WS distance and the $L_2$ norm measure fundamentally different error aspects. The former measures the "effort" to transport one distribution to another, and has units of distance. The latter measures the energy (squared difference) between vectors, and has units of the square of the quantity being measured. Thus, it is nontrivial to mix them when computing precise quantities. 
Nevertheless, we observe numerically that the simulation error $\epsilon_s$ in the $L_2$ norm exhibits the same convergence rate as in the $W_1$ distance.
\end{remark}

\subsubsection{Model error in the moment calculation $\epsilon_m$}
After solving the fluid model \eqref{eqn: fluid-model}, the moments of the diffusive part are calculated as \eqref{eqn: m0}-\eqref{eqn: m2}.
We now show that the error in these moment calculations is bounded by $\mathcal{O}(\varepsilon^2)$.
According to the derivation of the fluid model in Section \ref{sec: fluid model}, we know 
\begin{equation} \label{eqn: fluid error f}
    \lVert f(x,v,t)- f_f(x,v,t)\rVert_2 = \lVert\varepsilon^2 f_2(x,v,t) + \varepsilon^3 f_3(x,v,t) + \cdots\rVert_2 = \BO(\varepsilon^2).
\end{equation}
Here we use the fact that $f_2(x,v,t)$ is bounded by a constant independent of $\varepsilon$ \cite{maesHilbertExpansionBased2023a}. A rigorous proof of \eqref{eqn: fluid error f} can be found in \cite{othmerDiffusionLimitTransport2000b}. 
Further, we bound the error of moments \eqref{eqn: moments} as
\begin{equation}
    \epsilon_m = \left\lVert\intv v^l f(x,v,t) -  v^l f_f(x,v,t)dv\right\lVert_2 \leq \intv \lVert v^l\rVert_2 \left\lVert f(x,v,t)-f_f(x,v,t)\right\lVert_2dv \leq \BO(\varepsilon^2) \intv \lVert v^l\lVert_2 dv = \BO(\varepsilon^2),
\end{equation}
with $l=0, 1, 2$ corresponding to, respectively, $m_0(x), m_1(x)$ and $m_2(x)$ in \eqref{eqn: m0}-\eqref{eqn: m2}. 

\subsubsection{Time evolution error $\epsilon_{\dt}$}
In the fluid estimation procedure, we use the average duration of all diffusive steps for the evolution time of the fluid model \eqref{eqn: fluid-model}, as explained in Section \ref{sec: fluid estimaion}.    
Since the time of each diffusive step is less than $\dt$, the error for the averaging is bounded by $\BO(\dt)$. 
In addition, the averaging \eqref{eqn: fluid estimation time} is empirical, leading to an additional statistical error.    
Thus, the simulation time of the fluid model \eqref{eqn: fluid-model} has the error
    \begin{equation}\label{eqn: biased time}
        \delta_{\dt}=\BO(\dt)+\eta_{\dt},
    \end{equation}
where $\eta_{\dt}=\BO(1/\sqrt{N})$ is the statistical error, which is assumed to be negligible, due to the limited number of diffusive steps, as discussed in Section~\ref{sec: sub kdmc error}.
    
The biased simulation time \eqref{eqn: biased time} affects the solution of the fluid model as follows.
Let $\rho(x,t)$ be the true solution of fluid model \eqref{eqn: fluid-model} at time $t$, and $\rho(x,t+\delta_{\dt})$ is the solution simulated with the biased simulation time. Given the first-order Taylor expansion of $\rho(x, t+\delta_{\dt})$
\begin{equation}
    \rho(x, t+\delta_{\dt}) = \rho(x,t)+\delta_{\dt}\partial_t\rho(x,t) + \BO(\delta_{\dt}^2),
\end{equation}
the error due to the biased simulation time, which refers to the time evolution error, is then
\begin{equation}
    \epsilon_{\dt} = \lVert\rho(x, t+\delta_{\dt}) - \rho(x,t)\lVert_2 = \lVert\delta_{\dt}\partial_t\rho(x,t) + \BO(\delta_{\dt}^2)\lVert_2 = \lVert\delta_{\dt}\lVert_2 = \BO(\dt)+\eta_{\dt},
\end{equation}
where $\partial_t\rho(x,t)$ is independent of $\dt$ and $\eta_{\dt}$.   
For a homogeneous background, the averaging provides an unbiased estimate of the true evolution time since each diffusive time follows the same distribution $\text{Exp}(R)$ when $ R(x) = R$, that is, all samples are independent and identically distributed. 
    As a result, only a small statistical error is introduced. 

\subsubsection{Discretization errors $\epsilon_x$ and $\epsilon_{t}$}
Last, the numerical solution of the fluid model \eqref{eqn: fluid-model} introduces spatial and temporal discretization errors, denoted as $\epsilon_x$ and $\epsilon_t$, respectively. In this work, we use an upwind scheme for the advection term and a central difference scheme for the diffusive term, resulting in first-order accuracy in both space and time. 
Consequently, the spatial discretization error $\epsilon_x=\BO(\Delta x)$ with $\Delta x$ the spatial step size.
As the evolution time of this equation in the fluid estimation procedure is less than $\dt$, the time step used in the numerical solver will be much smaller than $\dt$. This ensures $\epsilon_t$ is $\BO(\dt)$. 
In practice, $\epsilon_t$ is negligible compared to $\epsilon_{\dt}=\eta_{\dt}+\BO(\dt)$, and we omit it from further discussion.

It is trivial to prove that the error of the fluid estimation is the sum of the above-listed errors, due to the linearity of the Boltzmann-BGK equation \eqref{eqn: BBGK} and the fluid model \eqref{eqn: fluid-model}. Therefore, we skip the proof and have
\begin{equation}\label{eqn: fluid est convergence}
    \epsilon_d  = \epsilon_m + \epsilon_{\dt}+\epsilon_x + \epsilon_i
    = \BO(\varepsilon^2) + \BO\left(\dt\right) + \BO(\Delta x) + \eta + \begin{cases} 
              \BO\left(\frac{\varepsilon^2}{\dt}\right),  & \text{if } \dt \gg \varepsilon^2, \\
              \BO\left(\frac{\dt}{\varepsilon^2}\right),  & \text{if } \dt \ll \varepsilon^2.
        \end{cases}
\end{equation}
where $\eta=\max(\eta_s,\eta_k,\eta_{\dt})=\BO(1/\sqrt{N})$ is the statistical error, assumed to be small.

We conclude our discussion of the simulation error by neglecting the discretization error and the influence of the number of particles $I$,
as our focus is on the performance of the KDMC method with respect to the time step $\dt$ and the diffusive scaling parameter $\varepsilon$.
Next, we substitute the error bound for the kinetic part \eqref{eqn: k error dt}, \eqref{eqn: k error e},
and that for diffusive part \eqref{eqn: fluid est convergence} into the error formula \eqref{eqn: case e_kd}, we obtain 
\begin{equation}\label{eqn: err kd dt fixed}
    \epsilon_{kd}(\varepsilon) =  
    \begin{cases} 
              \BO\left(\varepsilon^2\right),  & \text{if } \dt \gg \varepsilon^2, \\
              \BO\left(\frac{1}{\varepsilon}\right),  & \text{if } \dt \ll \varepsilon^2,
        \end{cases}
        \quad \text{ and } \quad
    \epsilon_{kd}(\dt) =    
    \begin{cases} 
              \BO\left(\frac{1}{\dt}\right),  & \text{if } \dt \gg \varepsilon^2, \\
              \BO\left(\dt\right),  & \text{if } \dt \ll \varepsilon^2,
        \end{cases}
\end{equation}
where \( \epsilon_{kd}(\varepsilon) \) denotes the estimation error when the time step \( \Delta t \) is fixed and the error depends only on the diffusive scaling parameter \( \varepsilon \), whereas \( \epsilon_{kd}(\Delta t) \) denotes the error when \( \varepsilon \) is fixed and the error depends only on \( \Delta t \).  
Note that, when $\dt \ll \varepsilon^2$, the kinetic part error $\eta$ in \eqref{eqn: k error dt} is expected to dominate $\epsilon_{kd}(\Delta t)$. However, since $\eta_k$ is assumed small, the diffusive part error, which is of order $\BO(\dt)$, remains significant. Consequently,  we have $\epsilon_{kd}(\Delta t)=\BO(\dt)$ if $\dt \ll \varepsilon^2$.


\section{Numerical experiments}\label{sec: Num Ex}
This section numerically validates the theoretical analysis presented in Sections \ref{sec: KDMC analysis} and \ref{sec: estimation sec}. The convergence of the KDMC simulation is examined in Section~\ref{sec: num kd}. 
For the estimation procedure, we introduce a test case in Section \ref{sec: test case}, with which the estimation errors of the kinetic and diffusive parts are then displayed separately in Sections~\ref{sec: est kinetic} and~\ref{sec: est diff}, respectively.
The total estimation error is presented in Section~\ref{sec: estimation error}.
All experiments in Section~\ref{sec: num kd}-~\ref{sec: est diff} make use of a homogeneous background. 
To illustrate the effectiveness of KDMC and our analysis in a more realistic context, a heterogeneous fusion-relevant test case is presented in Section \ref{sec: fusion test case}. 
Finally, the computational cost of KDMC is analyzed in Section \ref{sec: computational cost}.
Throughout this work, we specify the key parameters necessary to interpret our results and refer to our code repository \cite{tangPythonCodeThis2025} for further details.

\subsection{KDMC simulation error}\label{sec: num kd}
In Section \ref{sec: KDMC analysis}, we proved the local error bounds for the KDMC simulation with respect to the diffusive scaling parameter $\varepsilon$ \eqref{eqn: local error e} and the time step $\dt$ \eqref{eqn: local error dt}, 
along with the global error bound \eqref{eqn: kdmc error} under the $W_1$ distance metric. 

In numerical experiments, we use the kinetic simulation described in  Section~\ref{sec: kinetic simulation} as a reference solution, and compare the particle distributions from the kinetic and the KDMC simulation at a terminal time $\bar{t}$. 
All particles start from the initial position $x_0=0$ and the initial velocity $v_0$ following the Maxwellian distribution \eqref{eqn: maxwellian}. 
The collision rate $R$ and the variance $\sigma_p^2$ are both equal to $1/\varepsilon^2$ with the diffusive scaling parameter $\varepsilon\in[10^{-3}, 10^0]$, 
the mean velocity is $u_p=2$, and the time step is $\dt=\bar{t}/K$ with $K\in[10^0, 10^4]$.

The local error is the error made in a single time step. That is, in both the kinetic and KDMC simulations, particles evolve up to time $\dt$. 
The local error is obtained by calculating the $W_1$ distance \eqref{Def: WS-2} between the two resulting distributions at $\dt$.
This error is illustrated in Figure \ref{fig: local dt} with the number of particles $I=5\times10^7$. 
In Figure \ref{Fig: verify_local_e}, we fix $\dt=\bar{t}/5$ and plot the $W_1$ distance against $\varepsilon$. The result shows that the local error scales as $\BO(\varepsilon^2)$ as $\varepsilon\rightarrow 0$. 
When $\varepsilon=0.1$ is fixed and the time step varies, as demonstrated in Figure \ref{Fig: verify_local_t}, the error decreases as $\BO(\dt^2)$ as $\dt\rightarrow0$. These observations are consistent with the analytical order of the local errors \eqref{eqn: local error e} and \eqref{eqn: local error dt}, respectively.

To evaluate the global error, we let $10^7$ particles evolve up to the terminal time $\bar{t}=0.0275$ instead of performing just a single time step $\dt$.
The convergence is illustrated in Figure \ref{fig: global e}. 
For fixed $\dt$, Figures \ref{fig: dt=1} and \ref{fig: dt=25} show the global error decreases as $\BO(\varepsilon^2)$ when $\dt\gg\varepsilon^2$ and increases as $\BO(1/\varepsilon^2)$ when $\dt\ll\varepsilon^2$. The largest bias of the KDMC simulation is around $\dt=\varepsilon^2$ given the form of the error bound \eqref{eqn: kdmc error}. 
Similarly, for fixed $\varepsilon$, Figures~\ref{fig: global dt=0.01} and ~\ref{fig: global dt=0.006} demonstrate that the global error decreases as $\BO(\dt)$ when $\dt\ll\varepsilon^2$ and grows like $\BO(1/\dt)$ when $\dt\gg\varepsilon^2$. 
In some cases, such as around $\dt=10^{-5}$ in Figure~\ref{fig: global dt=0.01}, the observed decay is faster than $\BO(\dt)$, indicating that the bound in~\eqref{eqn: kdmc error} is not always sharp.  
Nevertheless, the simulation results generally align well with the theoretical prediction.
In Section \ref{sec: estimation sec}, we showed that there are other error terms bounded by 
$\BO(\varepsilon^2)$ and $\BO(\dt)$ in the bound of the total estimation error \eqref{eqn: err kd dt fixed}. 
Thus, there is no need to further sharpen the error bound.

\begin{figure}[h]
\centering
    \makebox[\textwidth]{
    \begin{subfigure}{0.45\textwidth}
    \centering    
    \includegraphics[trim=0 0 0 30, clip, width=\textwidth]{
    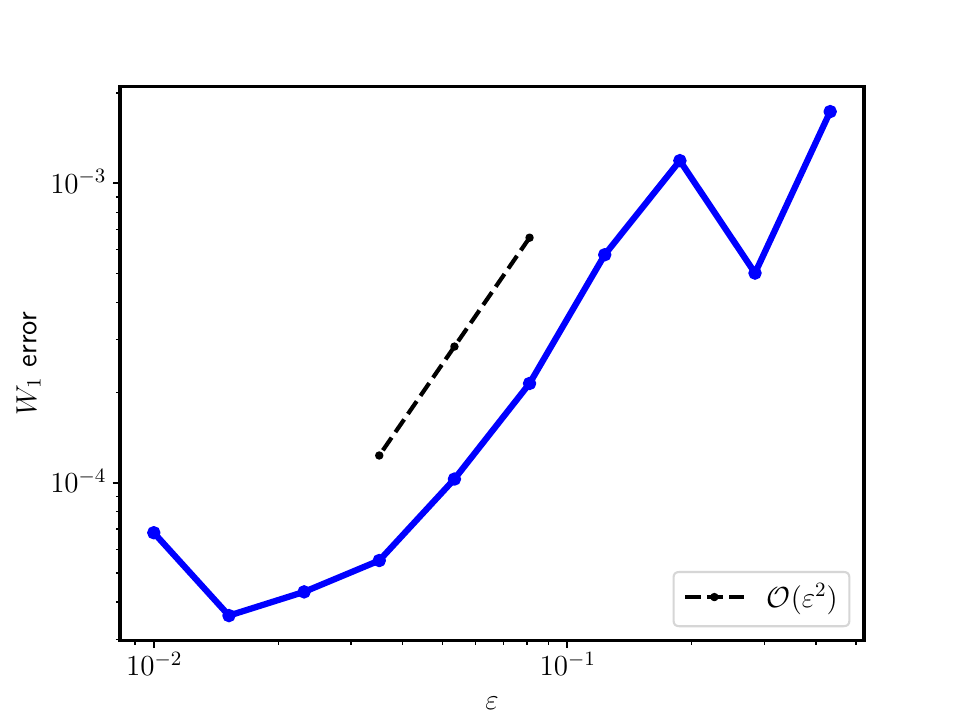}
    \caption{Fixed $\dt=\bar{t}/5$, varying $\varepsilon$}
    \label{Fig: verify_local_e}
    \end{subfigure}
    \begin{subfigure}{0.45\textwidth}
        \centering    
        \includegraphics[trim=0 0 0 30, clip, width=\textwidth]{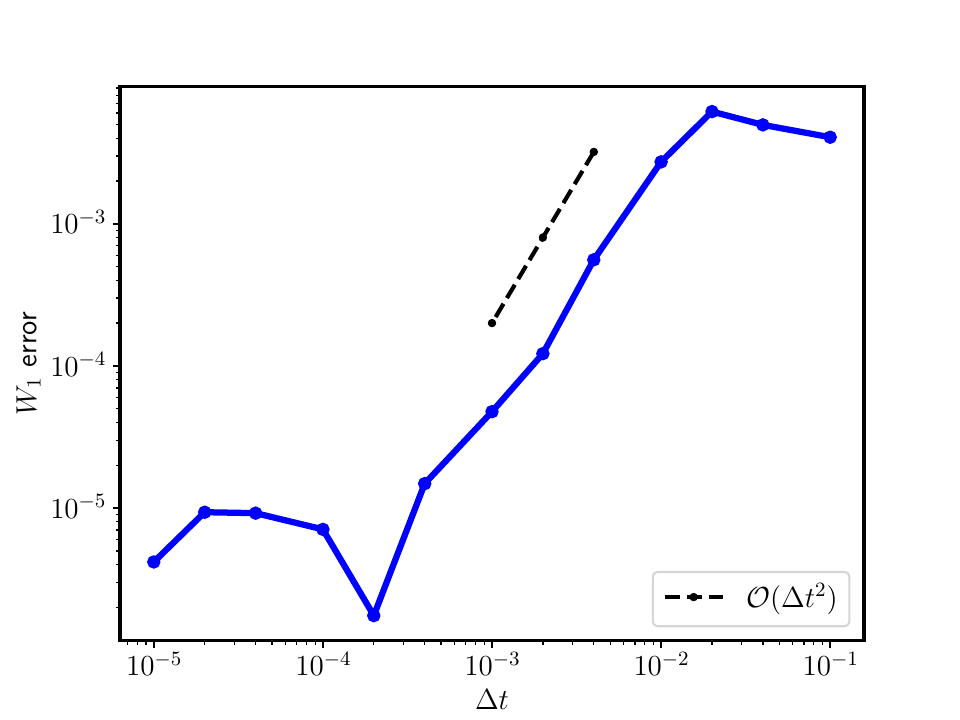}
        \caption{Fixed $\varepsilon=0.1$, varying $\dt$}
        \label{Fig: verify_local_t}
    \end{subfigure}
    }
    \caption{Local KDMC simulation error. Left: fixed $\dt$, the error decreases as $\BO(\varepsilon^2)$. Right: fixed $\varepsilon$, the error decreases as $\BO(\dt^2)$.}
    \label{fig: local dt}
\end{figure}

\begin{figure}[h]
\centering
    \makebox[\textwidth]{
     \begin{subfigure}{0.45\textwidth}
         \centering
         \includegraphics[trim=0 0 0 30, width=\textwidth]{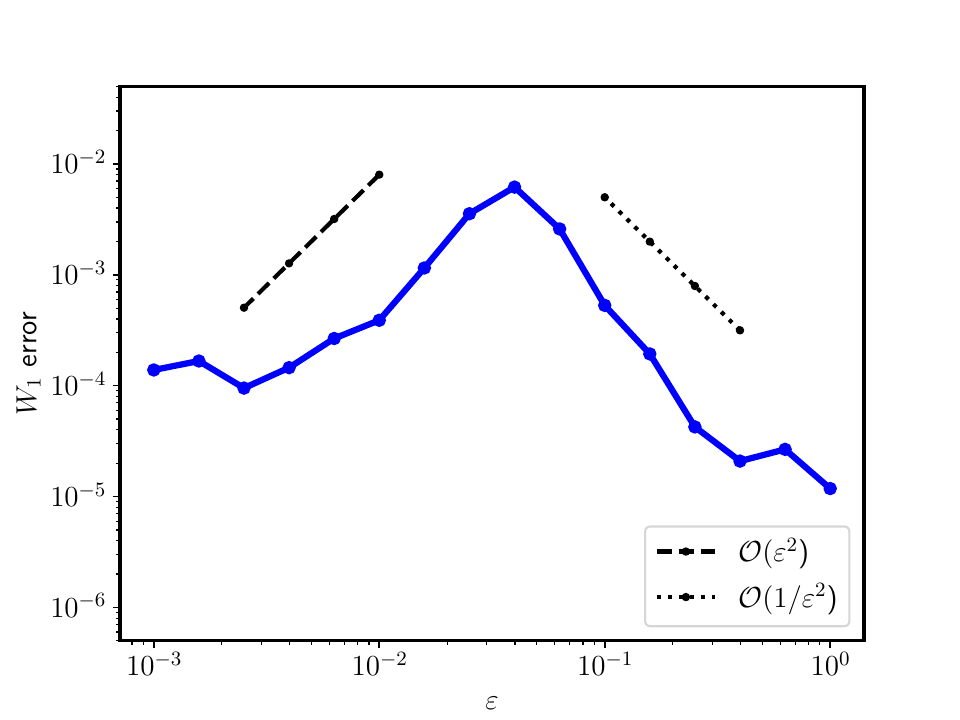}
         \caption{Fixed $\dt=\bar{t}/5$, varying $\varepsilon$}
         \label{fig: dt=1}
     \end{subfigure}
     \hspace{0.02\textwidth}
     \begin{subfigure}{0.45\textwidth}
         \centering
         \includegraphics[trim=0 0 0 30, width=\textwidth]{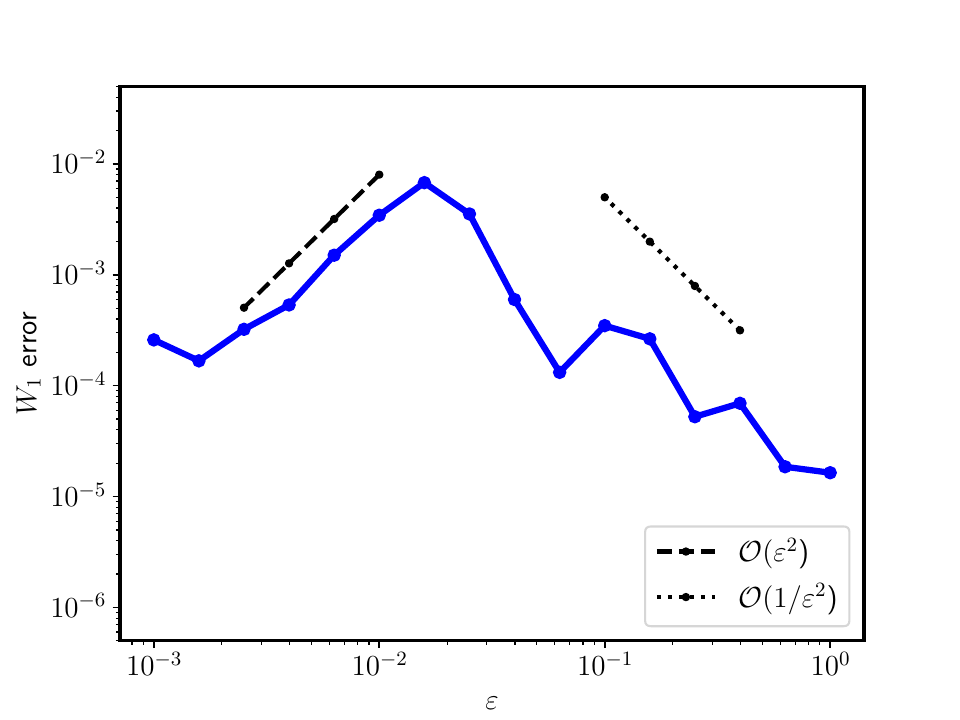}
         \caption{Fixed $\dt=\bar{t}/25$, varying $\varepsilon$}
         \label{fig: dt=25}
     \end{subfigure}
     }
     \hfill
     \centering
    \makebox[\textwidth]{
     \begin{subfigure}{0.45\textwidth}
         \centering
         \includegraphics[width=\textwidth]{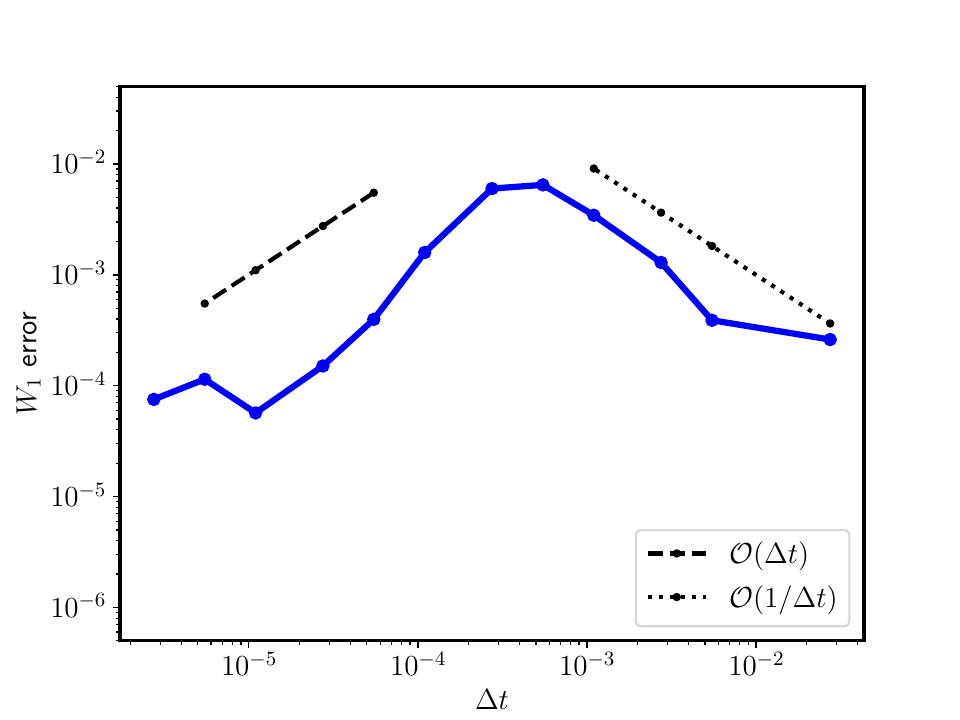}
         \caption{Fixed $\varepsilon=0.01$, varying $\dt$}
         \label{fig: global dt=0.01}
     \end{subfigure}
     \hspace{0.02\textwidth}
     \begin{subfigure}{0.45\textwidth}
         \centering
         \includegraphics[width=\textwidth]{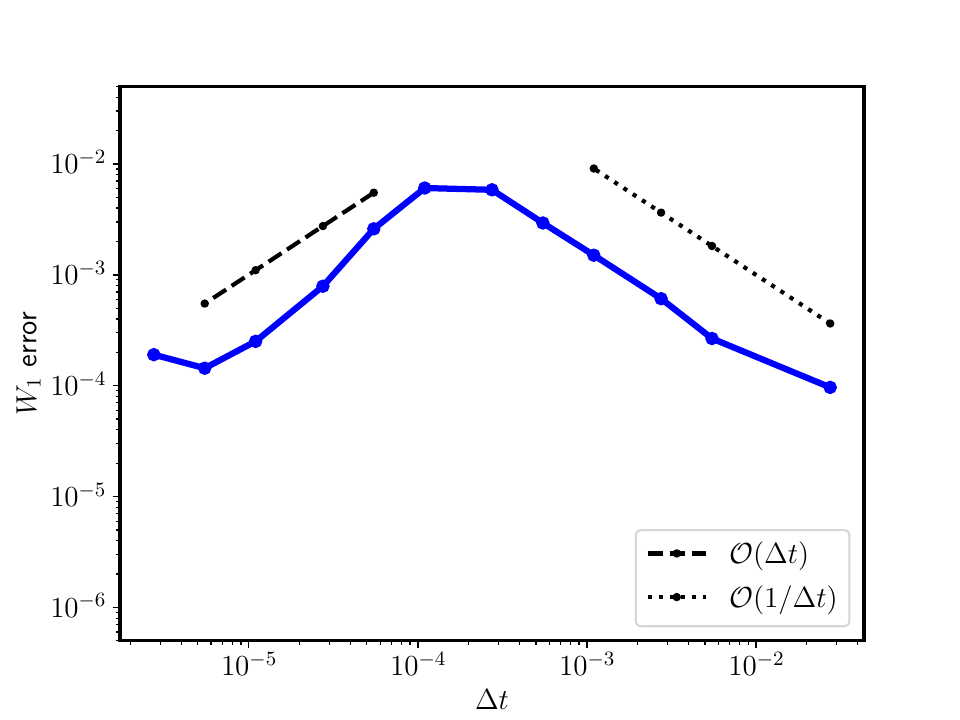}
         \caption{Fixed $\varepsilon\approx0.006$, varying $\dt$}
         \label{fig: global dt=0.006}
     \end{subfigure}
     }
        \caption{Global KDMC simulation error. First row: convergence with fixed $\dt$. Second row: convergence with fixed $\varepsilon$.}
        \label{fig: global e}
\end{figure}


\subsection{Homogeneous test case for estimation}\label{sec: test case}
In this test case, the particle system, defined over the periodic domain $x \in [0, 1]$, evolves over the time interval $t\in[0,\bar{t}]$ with $\bar{t} = 0.0275$.
The variance of the plasma velocity and the collision rate are
\begin{equation}
   \quad \sigma_p^2(x) = \frac{1}{\varepsilon^2}, \quad \text{and} \quad R(x) = \frac{1}{\varepsilon^2},
\end{equation}
respectively. The mean velocity is $u_p=30$ unless stated otherwise. The parameter $\varepsilon$ ranges from $0.01$ to $1$, implying a maximum collision rate of $10^4$. 
As the initial condition,
particles have position $x_0$ sampled from the density
\begin{equation}
    \rho(x, t=0) = 1 + \frac{1}{2\pi}\sin(2\pi x), \quad x\in[0, 1],
\end{equation}
and velocities following the Maxwellian distribution $M(v|x_0)$ given by \eqref{eqn: maxwellian}. A periodic boundary condition is imposed on the spatial domain.

In the numerical experiments, the three moments \eqref{eqn: moments} are discretized spatially as mentioned in Section \ref{sec: fluid estimaion}. 
If the reference moments are \( m_l=[m_{l,1}, \cdots, m_{l,J}]\), and the approximate moments are \( \hat{m}_l=[\hat{m}_{l,1}, \cdots, \hat{m}_{l,J}]\), with $l=0, 1, 2$, the relative \(L_2\)  error $E$ with respect to the reference moments is thus calculated as
\begin{equation}\label{eqn: rel l2 norm}
    E = \frac{\left(\sum_{j=1}^J \left(m_{l,j}-\hat{m}_{l,j}\right)^2\right)^{1/2}}{\left(\sum_{j=1}^J m_{l,j}^2\right)^{1/2}}.
\end{equation}

\subsection{Estimation error: kinetic part}\label{sec: est kinetic}
We now present the numerical results for the estimation errors under the relative $L_2$ metric \eqref{eqn: rel l2 norm}, starting with the error in the kinetic part as defined in \eqref{eqn: e_k e_d def}. 
To calculate the error \eqref{eqn: e_k e_d def}, we first need to isolate the kinetic part of the estimated moment $m^{kd,k}$ from the total moment $m^{kd}$ and construct the reference solution $\alpha m^{k}$.

\begin{figure}[h]
\centering
    \makebox[\textwidth]{
     \begin{subfigure}{0.45\textwidth}
         \centering
         \includegraphics[trim=0 0 0 25, width=\textwidth]{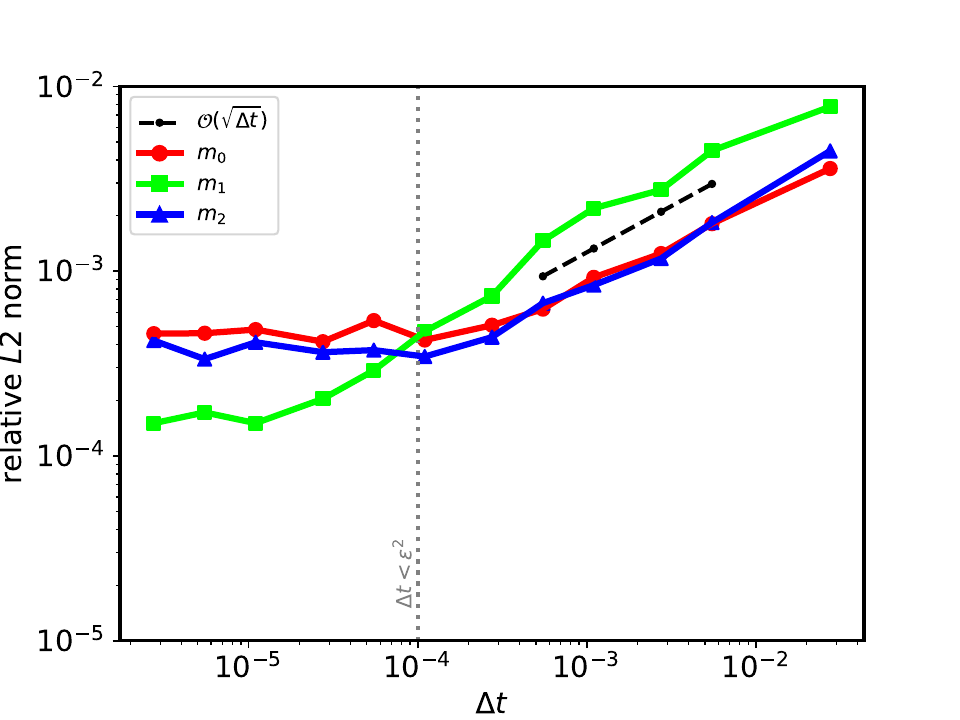}
         \caption{Fixed $\varepsilon=0.01$, varying $\dt$}
         \label{fig: k error up30 fixed e}
     \end{subfigure}
     \hspace{0.02\textwidth}
      \begin{subfigure}{0.45\textwidth}
         \centering
         \includegraphics[trim=0 0 0 25, width=\textwidth]{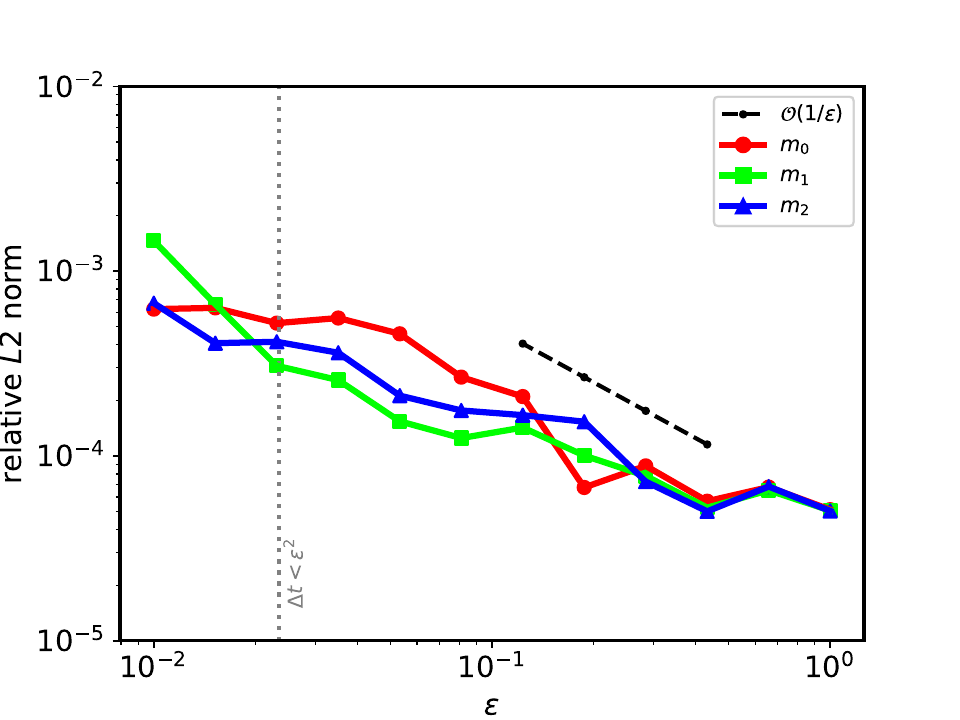}
         \caption{Fixed $\dt=0.0275/50$, varying $\epsilon$}
         \label{fig: k error up30 fixed dt}
     \end{subfigure}
     }
     \captionsetup{justification=centerlast}
     \caption{Kinetic part estimation error with $u_p=30$.}
     \label{fig: k part error up30}
\end{figure}

\begin{figure}[h]
\centering
    \makebox[\textwidth]{
     \begin{subfigure}{0.45\textwidth}
         \centering
         \includegraphics[trim=0 0 0 25, width=\textwidth]{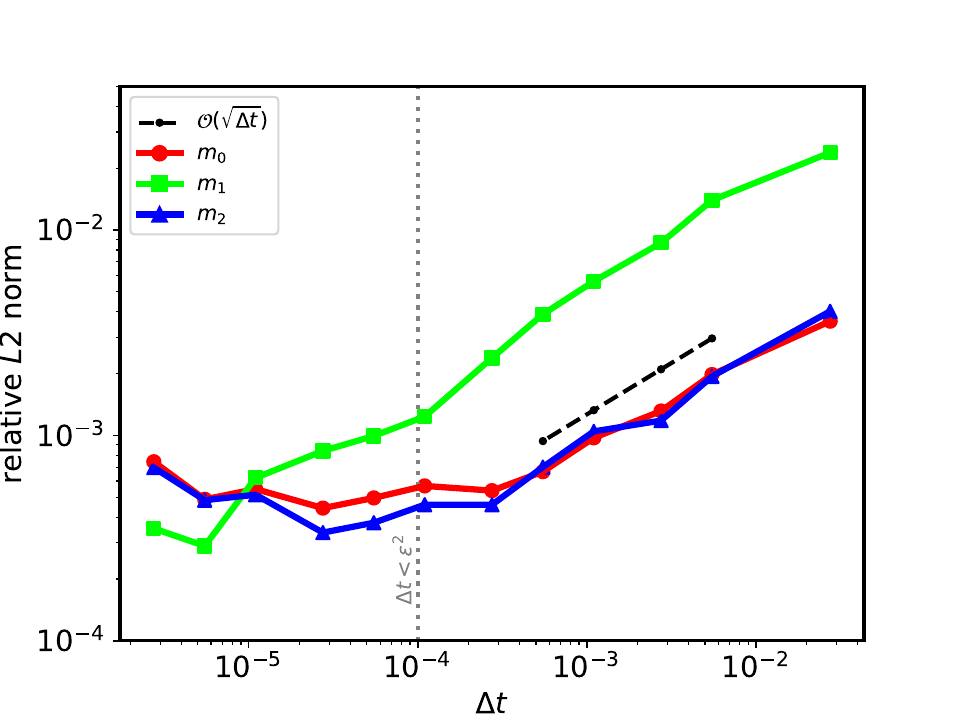}
         \caption{Fixed $\varepsilon=0.01$, varying $\dt$}
         \label{fig: k error up10 fixed e}
     \end{subfigure}
     \hspace{0.02\textwidth}
     \begin{subfigure}{0.45\textwidth}
         \centering
         \includegraphics[trim=0 0 0 25, width=\textwidth]{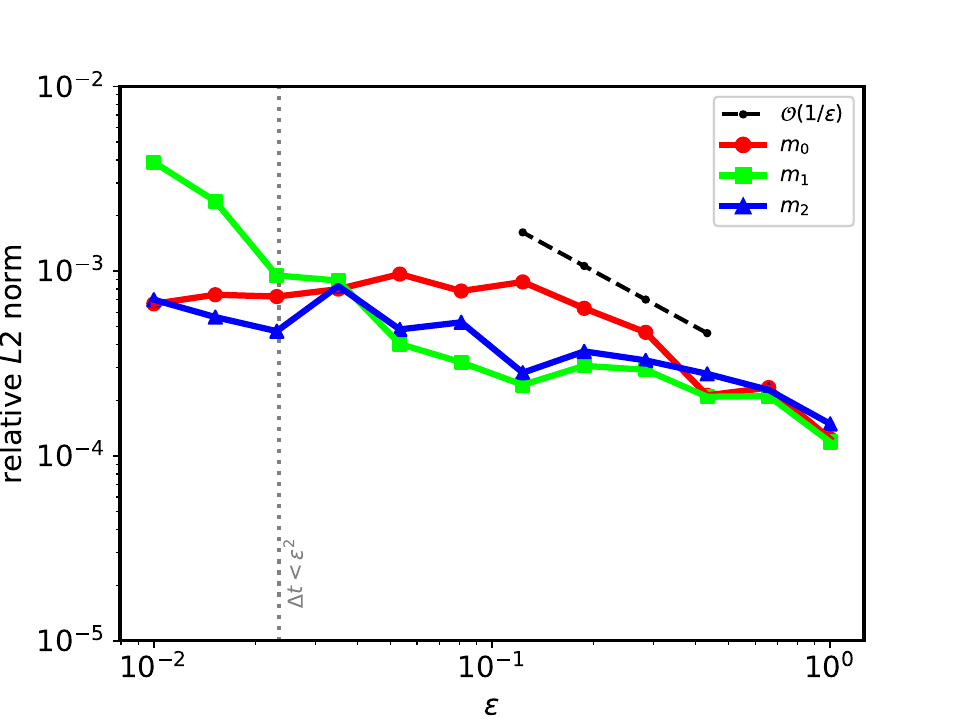}
         \caption{Fixed $\dt=0.0275/50$, varying $\varepsilon$}
         \label{fig: k error up10 fixed dt}
     \end{subfigure}
     }
     \captionsetup{justification=centerlast}
     \caption{Kinetic part estimation error with $u_p=10$. A smaller $u_p$ leads to a larger error of the $m_1$ moment}
     \label{fig: k part error up10}
\end{figure}

Due to the hybrid procedure of kinetic and diffusion motions in KDMC, isolating the kinetic part is not straightforward, since the $k$-th kinetic step is inherently influenced by the $i=0,\ldots,k-1$-th diffusive steps. 
To address this and obtain the approximate moment $m^{kd,k}$, we let the diffusive steps be mimicked using kinetic motions that we do not score. 
Specifically, within each $\dt$ time segment, the first kinetic step is scored by the MC estimator \eqref{eqn: K estimator}. 
Instead of a diffusive step \eqref{eqn: SDE update}, particles move kinetically as \eqref{eqn: a kinetic step} for the rest of the time without scoring. 
In this way, the simulation only has a negligible statistical error, and the estimation error of the diffusive part is removed, as the kinetic motion is unbiased and no diffusive motion exists.  
(Note that forcing particles to stop at $\dt$ with kinetic steps is unbiased, thanks to the memorylessness property of the exponential distribution for sampling the free-flight time $\tau$ used in Section \ref{sec: kinetic simulation}.) 
The reference moment $\alpha m^k$ is obtained by repeating the above process with different random seeds, since the kinetic simulation is used throughout, and scoring only the kinetic part (the first step within a time step) recovers the reference $\alpha m^{k}$. 
The error of the kinetic part is then calculated by comparing $m^{kd,k}$ and $m^k$ under the relative $L_2$ metric \eqref{eqn: rel l2 norm}. 



Figure \ref{fig: k part error up30} shows the kinetic error obtained in this manner with $10^7$ particles in both experiments for the reference and approximation. 
In all figures that follow, red, green, and blue represent the error of the moments $m_0, m_1$, and $m_2$, respectively. 
The vertical dashed line indicates the point $\dt=\varepsilon^2$ where the error is expected to stagnate.
Figure \ref{fig: k error up30 fixed e} presents the relative $L_2$ error against the time step $\dt$ with the scaling parameter $\varepsilon=0.01$. 
The error scales as $\BO(\sqrt{\dt})$ initially and stagnates when $\dt\ll\varepsilon^2$. 
The stagnation starts around $\dt=\varepsilon^2$ as expected.
Next, we fix $\dt=\bar{t}/50$ (i.e., $50$ time steps) and vary $\varepsilon$ in Figure \ref{fig: k error up30 fixed dt}. 
The error grows at the rate $\BO(1/\varepsilon)$ as $\varepsilon$ decreases, and stagnates when $\dt\gg\varepsilon^2$. 
The stagnation starts again around $\dt=\varepsilon^2$. 
In both cases, the stagnation occurs at roughly the same magnitude as predicted in Section \ref{sec: kinetic error e}.
In total, the results are consistent with the analytical error \eqref{eqn: k error dt} and \eqref{eqn: k error e} in Section \ref{sec: error of kinetic part}.

As described in Section \ref{sec: error of kinetic part}, the $m_1$ moment has a larger error than the $m_0$ and $m_2$ moments if $u_p$ is small. To show this, let $u_p=10$ instead of $30$, and repeating the experiment above, we get the convergence plot shown in Figure \ref{fig: k part error up10}. It can be seen that the $m_0$ and $m_2$ moments are similar to those shown in Figure \ref{fig: k part error up30}, but the $m_1$ moment has a greater error in the regime $\dt\gg\varepsilon^2$.

\subsection{Estimation error: diffusive part}\label{sec: est diff}
In this section, we study the estimation error of the diffusive part numerically. 
We first illustrate its components: 
the model error \( \epsilon_m \), the spatial discretization error \( \epsilon_{x} \) in Section \ref{sec: diff part err 1} and the time discretization error \( \epsilon_{\dt} \) in Section \ref{sec: diff part err 2}. 
Subsequently, the estimation error of the diffusive part, including the error propagated from the initial error $\epsilon_i$, is illustrated in Section \ref{sec: diff part err 3}. 
The corresponding theoretical analysis can be found in Section \ref{sec: num. error of fluid model}.
As before, the three colors: red, green, and blue represent the error of the moments $m_0, m_1$, and $m_2$, respectively, in all plots. 

\subsubsection{Model error $\epsilon_m$ and discretization error $\epsilon_{x}$}\label{sec: diff part err 1}
Estimating the moments using the fluid model~\eqref{eqn: fluid-model}, which approximates the Boltzmann-BGK equation~\eqref{eqn: BBGK}, introduces a model error $\epsilon_m=\BO(\varepsilon^2)$. Moreover, the spatial discretization error $\epsilon_{x}$ may dominate when the solver of the fluid model is not accurate enough. 
In the numerical experiments, the approximate moments are estimated as \eqref{eqn: m0}-\eqref{eqn: m2} by solving the fluid model in Section \ref{sec: fluid model} via the upwind scheme.
To study the effect of spatial discretization, we use two mesh resolutions: $200$ and $1000$ grid cells. This allows us to observe the discretization error $\epsilon_{x}$.
For reference, we take the moments estimated by the kinetic method
with $N=10^8$ particles and $1000$ grid cells.
If fewer particles were used, the second-order convergence in space would be obscured by the statistical error in the reference solution.

The numerical results for both grid resolutions are presented in Figure \ref{fig: fluid model error}.
In both cases, we see that 
when $\varepsilon$ is large, the fluid model is not an accurate approximation, and the error does not decrease significantly. 
As $\varepsilon$ decreases, the expected second-order convergence, i.e., $\BO(\varepsilon^2)$, becomes evident. 
However, for $200$ cells, the error of $m_1$ stagnates near $\varepsilon=0.01$, indicating that another error dominates. This dominant error is the discretization error. 
This fact is confirmed in Figure~\ref{fig: fluid model error 1000 cells}, where increasing the number of cells to $1000$ reduces the error of $m_1$ by nearly an order of magnitude, and the errors of the other two moments also reduce slightly.

The obvious discretization error in $m_1$ moment is due to the spatial derivative term in its definition \eqref{eqn: m1}. 
In contrast, $m_0$ contains no derivative term, making it less sensitive to spatial discretization. Although the $m_2$ moment \eqref{eqn: m2} does have the derivative term, its non-derivative term has a much greater value when $\varepsilon$ is small. Consequently, the relative contribution of the discretization error is less significant for $m_2$.
\begin{figure} [h]
\centering
    \makebox[\textwidth]{
     \begin{subfigure}{0.45\textwidth}
         \centering
         \includegraphics[trim=0 0 0 25, width=\textwidth]{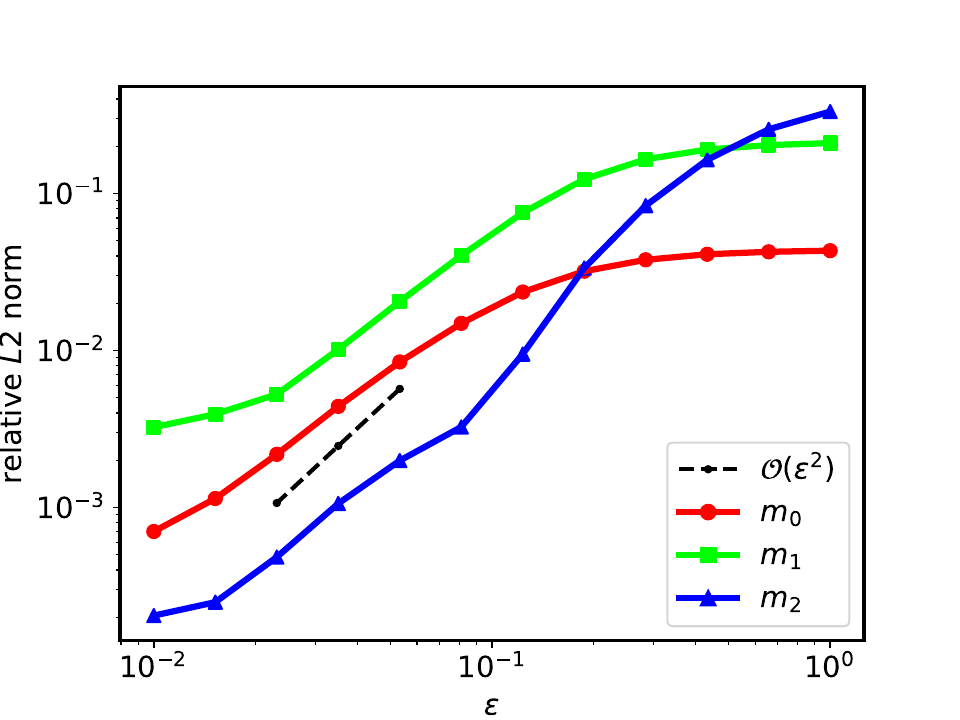}
         \caption{200 cells}
         \label{fig: fluid model error 200 cells}
     \end{subfigure}
        \hspace{0.02\textwidth}
     \begin{subfigure}{0.45\textwidth}
         \centering
         \includegraphics[trim=0 0 0 25, width=\textwidth]{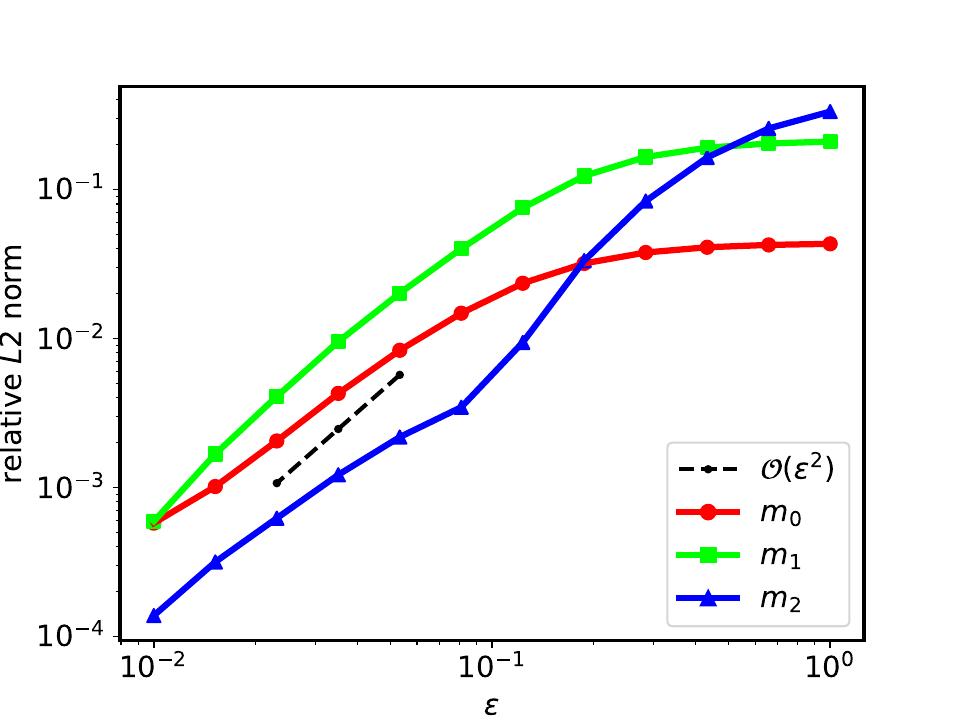}
         \caption{1000 cells}
         \label{fig: fluid model error 1000 cells}
     \end{subfigure}
     }
     \captionsetup{justification=centerlast}
     \caption{Model error \(\epsilon_m\) and discretization error \(\epsilon_{\Delta x}\). From \(200\) cells to \(1000\) cells, the error of the \(m_1\) moment at \(\varepsilon = 0.01\) decreases significantly, indicating the domination of the discretization error.}
     \label{fig: fluid model error}
\end{figure}

\subsubsection{Time evolution error $\epsilon_{\dt}$}\label{sec: diff part err 2}
As analyzed in Section \ref{sec: num. error of fluid model}, only a statistical error exists when approximating the evolution time of the fluid estimation by the averaging with a homogeneous background, i.e., $\epsilon_{\dt}=\eta=\BO(1/\sqrt{I})$, where $I$ is the number of particles.
To verify $\epsilon_{\dt}=\eta$ numerically, we first isolate the error of the diffusive part from the total estimation error, 
and then isolate the time evolution error $\epsilon_{\dt}$ from the error of the diffusive part. 
In the diffusive part, the error can be decomposed as $\epsilon_d = \epsilon_m + \epsilon_{\dt} + \epsilon_x  +\epsilon_{i}$ (see \eqref{eqn: fluid est convergence}). 
Among these, only $\epsilon_{\dt}$ and the propagated error $\epsilon_i$ depend on the time $t$. We then need to remove the propagated error from the KDMC simulation $\epsilon_i$ for focusing solely on
$\epsilon_{\dt}$. 
This can be done similarly to the strategy used in Section \ref{sec: est kinetic}, that is, mimicking the diffusive motion with the kinetic motion. 

In particular, to construct the reference solution, we fix the time step size $\dt$ in the kinetic method as in the experiment in Section \ref{sec: est kinetic}. 
In each $\dt$, the first flight is not scored, but the remaining mimicked diffusive part is scored. 
As a result, we obtain the moments only contributed by the diffusive part, which is thus the reference.
For the approximation solution, we use the KDMC simulation, but replace the diffusive motion with kinetic motions, which turns out to be another kinetic simulation with the fixed $\dt$. 
Then, we assume only the starting position of a diffusive step is known as the real diffusive motion, 
so that the fluid estimation described in Section \ref{sec: fluid estimaion} can be applied. 
In this setting, there is no propagated error $\epsilon_i$, since all particle movements follow the unbiased kinetic simulation. By scoring only the diffusive part, we get the separated error $\epsilon_{\dt}$, in terms of $\dt$.

 \begin{figure}[h]
 \centering
    \makebox[\textwidth]{
     \begin{subfigure}{0.45\textwidth}
         \centering
         \includegraphics[trim=0 0 0 25, width=\textwidth]{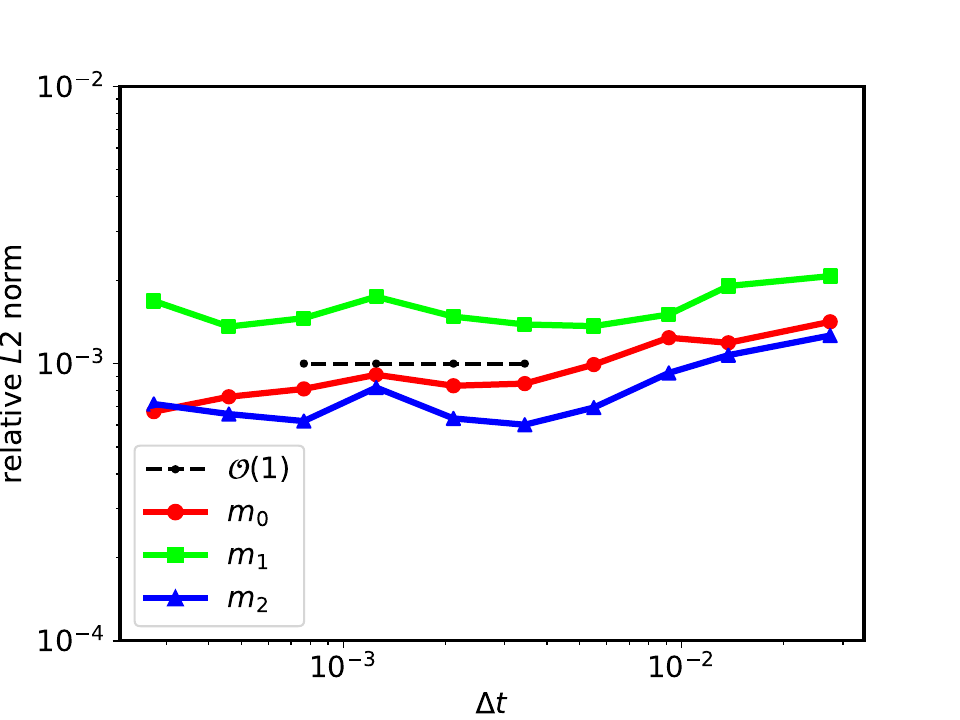}
         \caption{Fixed $\varepsilon=0.01$, varying $\dt$}
         \label{fig: dt error ep}
     \end{subfigure}
     \hspace{0.02\textwidth}
     \begin{subfigure}{0.45\textwidth}
         \centering
         \includegraphics[trim=0 0 0 25, width=\textwidth]{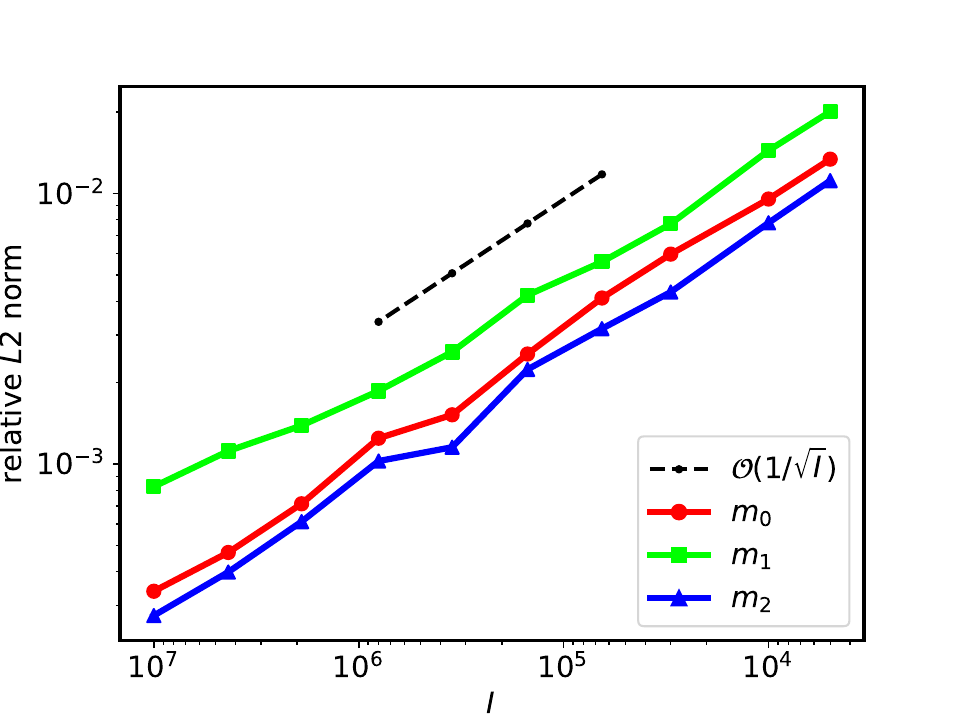}
         \caption{Fixed $\dt=0.0275/30$, varying the number of particles $I$}
         \label{fig: dt error ndt}
     \end{subfigure}
     }
     \captionsetup{justification=centerlast}
     \caption{The time evolution error $\epsilon_{\dt}$. This error is of order $\BO(1/\sqrt{N})$ and not depend on $\dt$.}
     \label{fig: e_dt error}
\end{figure}

The numerical results are displayed in Figure \ref{fig: e_dt error}.
Figure~\ref{fig: dt error ep} presents the relative $L_2$ error \eqref{eqn: rel l2 norm} between the constructed reference solution and the approximate solution described above as a function of the time step \( \dt \).  
It shows that the error \( \epsilon_{\dt} = \mathcal{O}(1) \) with respect to \( \dt \), confirming that it is independent of the time step.  
Furthermore, plotting the error against the number of particles $I$ in Figure \ref{fig: dt error ndt}, we can see that $\epsilon_{\dt}=\eta=\BO(1/\sqrt{I})$.
The two experiments confirm that $\epsilon_{\dt}$ is purely a statistical error and does not depend on the choice of time step $\dt$.

\subsubsection{Diffusive error $\epsilon_d$}\label{sec: diff part err 3}
We have shown the components $\epsilon_m, \epsilon_{x}$, and $\epsilon_{\dt}$ of the diffusive part of the estimation error in Sections~ \ref{sec: diff part err 1} and ~\ref{sec: diff part err 2}.
Here, we illustrate this error $\epsilon_d$ as defined in \eqref{eqn: e_k e_d def}, and simultaneously examine the propagated error $\epsilon_i$.
In \eqref{eqn: e_k e_d def}, the approximate moment $m^{kd,d}$ is the moment without scoring the kinetic part in KDMC.
The reference moment $(1-\alpha)m^{k}$, given by the kinetic method, is obtained by fixing the time step $\dt$ and without scoring the first flight.

The estimation error of the diffusive part is provided in \eqref{eqn: fluid est convergence}. 
For focusing on the effects of \( \dt \) and \( \varepsilon \), we employ a sufficiently large number of particles with \( I=10^7 \) and a spatial mesh with 200 grid cells, 
ensuring that both the statistical and discretization errors are negligible. 
Hence, only the model error \( \epsilon_m \) and the propagated error \( \epsilon_i \) remain significant.
The error $\epsilon_d$ is illustrated in Figure \ref{fig: diffusive part error up30}. 
Let $\varepsilon=0.01$, the error as a function of $\dt$ is shown in Figure \ref{fig: diffusive error ep}. 
The vertical dashed line indicates the point $\dt=\varepsilon^2$.
It can be seen that the error scales as $\BO(1/\dt)$ when $\dt\gg\varepsilon^2$, and as $\BO(\dt)$ when $\dt\ll\varepsilon^2$, which displays the convergence behavior of $\epsilon_i$ in \eqref{eqn: initial error} and confirms the existence of $\epsilon_i$, since other error components $\epsilon_m, \epsilon_{x}$, and $\epsilon_{\dt}$ do not have this shape of convergence.
Next, we fix $\dt =\Bar{t}/5$ and plot the error against $\varepsilon$ in Figure \ref{fig: diffusive error ndt}. 
The error curves exhibit a similar shape to the model error displayed in Figure \ref{fig: fluid model error 200 cells}, i.e., the error decreases only when $\varepsilon$ is small. 
This indicates that the model error $\epsilon_k$ dominates.

In Figure \ref{fig: diffusive error ep}, the error of the $m_1$ moment does not decrease as $\BO(\dt)$ in the regime $\dt\ll\varepsilon^2$ due to the domination of the error in the reference solution given by the kinetic method. If a larger mean velocity $u_p$ is used, the error of $m_1$ will decrease, such as $u_p= 100$ (see \cite{tangPythonCodeThis2025} for this case). 
The effect of $u_p$ for the kinetic simulation is briefly mentioned in Section \ref{sec: error of kinetic part} and demonstrated in Section \ref{sec: est kinetic}.
To maintain consistency and avoid repetition, we still use the test case with $u_p=30$ as in the other experiments.

 \begin{figure}[h]
 \centering
    \makebox[\textwidth]{
     \begin{subfigure}{0.45\textwidth}
         \centering
         \includegraphics[trim=0 0 0 25, width=\textwidth]{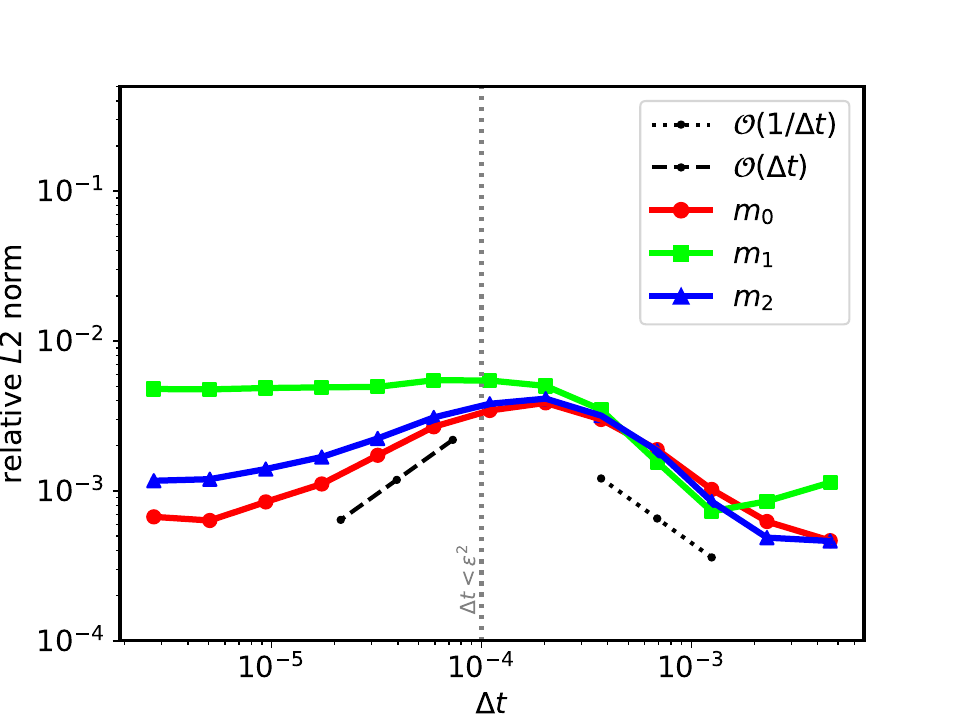}
         \caption{Fixed $\varepsilon=0.01$, varying $\dt$}
         \label{fig: diffusive error ep}
     \end{subfigure}
     \hspace{0.02\textwidth}
     \begin{subfigure}{0.45\textwidth}
         \centering
         \includegraphics[trim=0 0 0 25, width=\textwidth]{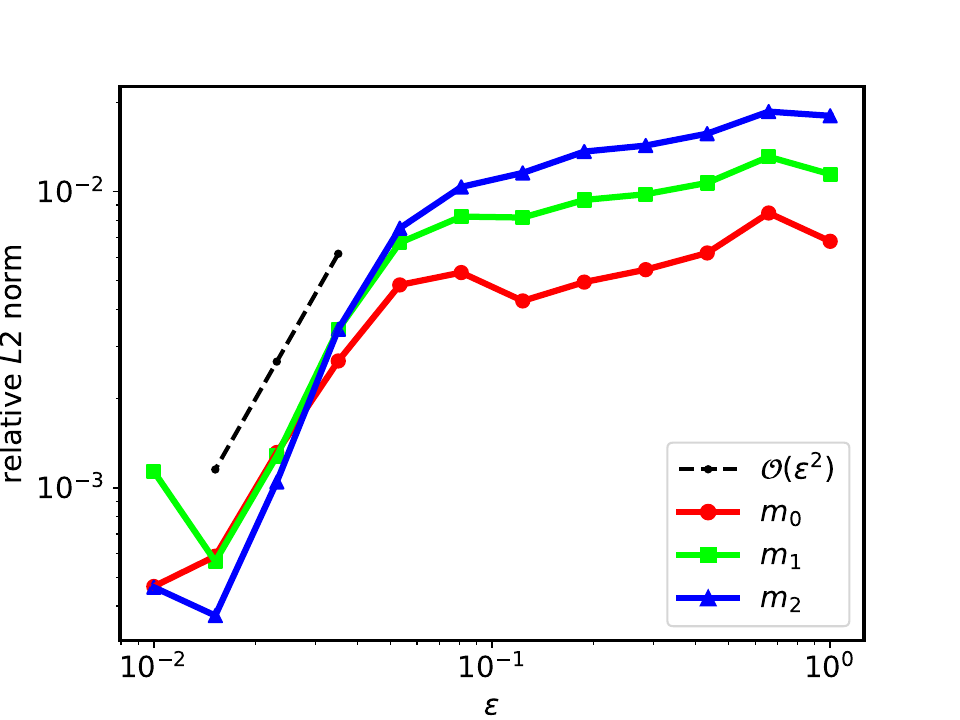}
         \caption{Fixed $\dt=\Bar{t}/5$ with $\Bar{t}=0.0275$, varying $\varepsilon$}
         \label{fig: diffusive error ndt}
     \end{subfigure}
     }
     \captionsetup{justification=centerlast}
     \caption{The diffusive part estimation error $\epsilon_d$.}
     \label{fig: diffusive part error up30}
\end{figure}
\subsection{Estimation error}\label{sec: estimation error}
We now present numerical results for the total estimation error of KDMC, defined in \eqref{eqn: e_kd}, with its error bound given in \eqref{eqn: err kd dt fixed}.
The presentation is organized to separate the effects of the time step \( \dt \) and the diffusive scaling parameter \( \varepsilon \).  
KDMC is expected to be more accurate than the fluid method described in Section \ref{sec: fluid model}, so the fluid method is compared in this experiment. 
The kinetic method is again the reference solution.
In all plots, solid lines represent the estimation errors of KDMC, while dashed lines correspond to the errors of the fluid method. 
As before, red, green, and blue represent the error of $m_0, m_1$, and $m_2$ moments, respectively. 
The vertical dashed line indicates the point where $\dt=\varepsilon^2$. 

In terms of the scaling parameter $\varepsilon$, the total estimation error $\epsilon_{kd}(\varepsilon) = \BO(1/\varepsilon)$ when $\dt\ll\varepsilon^2$, and $\epsilon_{kd}(\varepsilon) = \BO(\varepsilon^2)$ when $\dt\gg\varepsilon^2$.
The corresponding convergence is shown in Figure \ref{fig: est error e}, with $K=\bar{t}/\dt=1,5,25,$ and $75$ time steps used in Figures
 \ref{fig: est error n=1}-\ref{fig: est error n=75}, respectively.
If the time step $\dt$ is large such that $\dt\gg\varepsilon^2$, e.g., on the left side of the vertical line in Figures \ref{fig: est error n=1} and \ref{fig: est error n=5}, the diffusive part dominates KDMC, and the error convergence is of order $\BO(\varepsilon^2)$.  
Especially, if only one time step is used, i.e., $\dt=\bar{t}$, the dashed lines and the solid lines coincide when $\varepsilon\rightarrow0$ since KDMC converges to the fluid method.  
If the time step $\dt$ is small such that $\dt\ll\varepsilon^2$, e.g., on the right side of the vertical line in Figures \ref{fig: est error n=25} and \ref{fig: est error n=75}, 
the kinetic part dominates KDMC, and the error increases as $\BO(1/\varepsilon)$. 
Due to the domination of the kinetic part, where the fluid model \eqref{eqn: fluid-model} is not accurate, KDMC has a significantly lower error than the fluid method.

\begin{figure}[h]
    \centering
    \begin{subfigure}[b]{0.45\textwidth}
        \centering
        \includegraphics[trim=0 0 0 25, width=\textwidth]{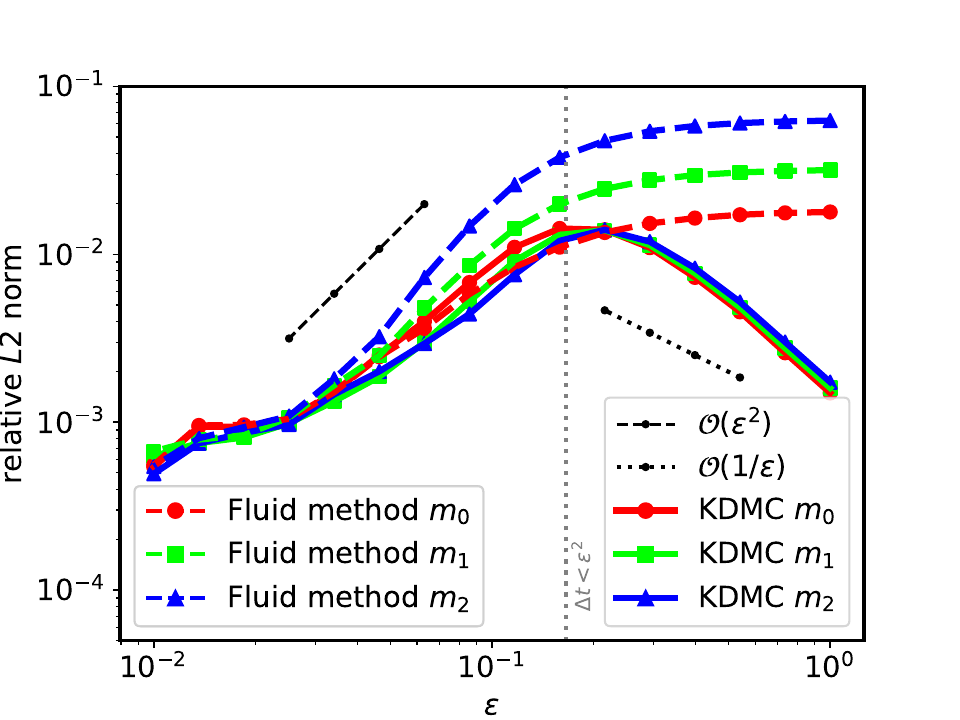}
        \caption{Fixed $\dt=\bar{t}/1$, varying $\varepsilon$}
        \label{fig: est error n=1}
    \end{subfigure}
    \begin{subfigure}[b]{0.45\textwidth}
        \centering
        \includegraphics[trim=0 0 0 25, width=\textwidth]{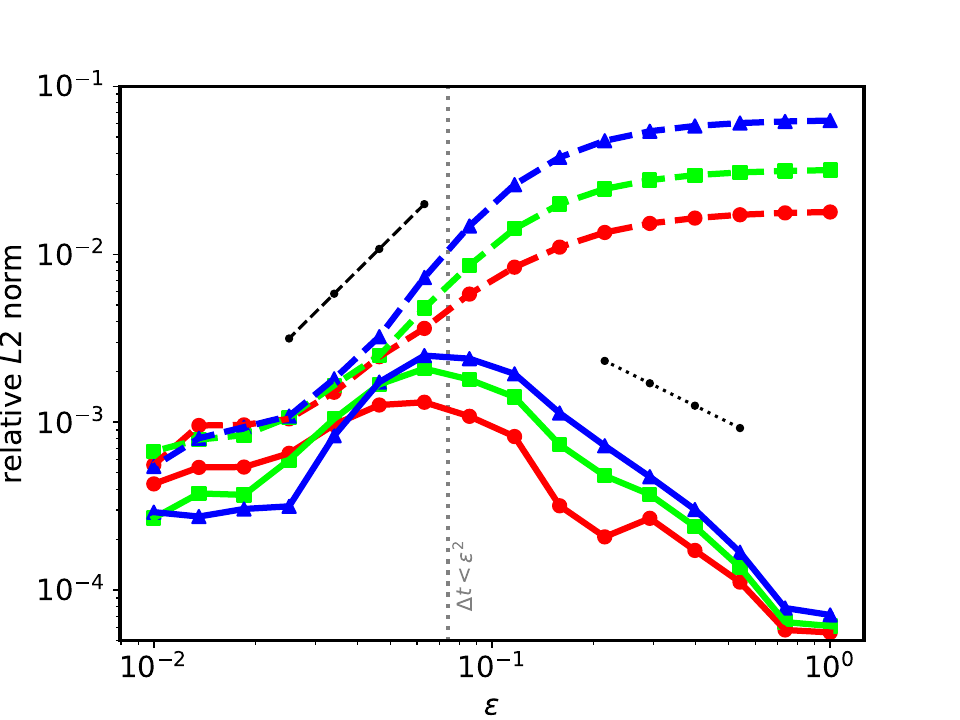}
        \caption{Fixed $\dt=\bar{t}/5$, varying $\varepsilon$}
        \label{fig: est error n=5}
    \end{subfigure}


    \begin{subfigure}[b]{0.45\textwidth}
        \centering
        \includegraphics[width=\textwidth]{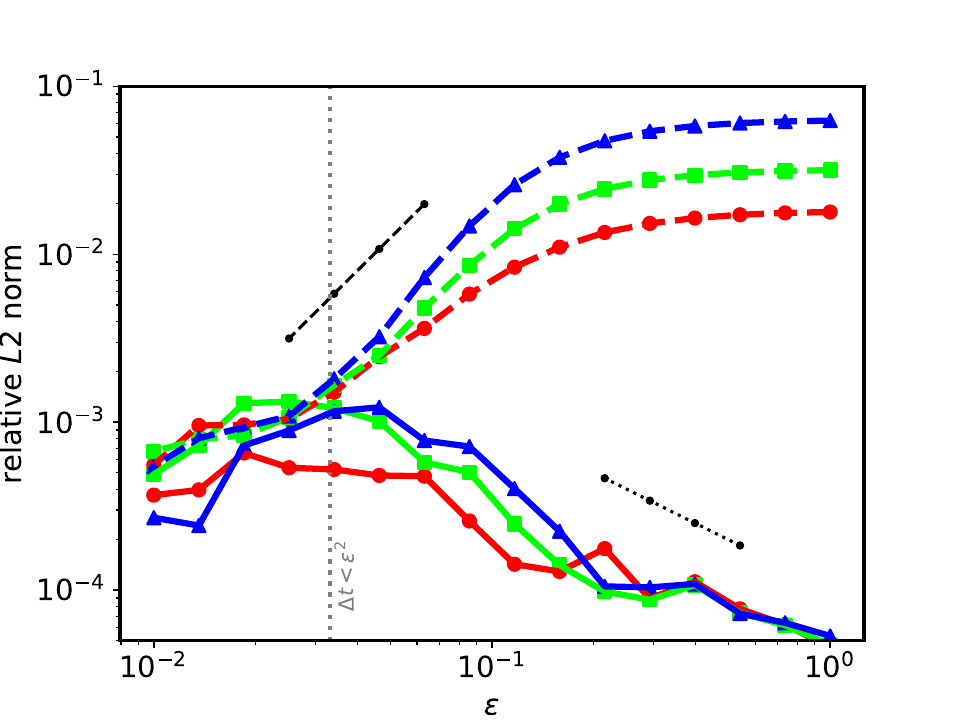}
        \caption{Fixed $\dt=\bar{t}/25$, varying $\varepsilon$}
        \label{fig: est error n=25}
    \end{subfigure}
    \begin{subfigure}[b]{0.45\textwidth}
        \centering
        \includegraphics[width=\textwidth]{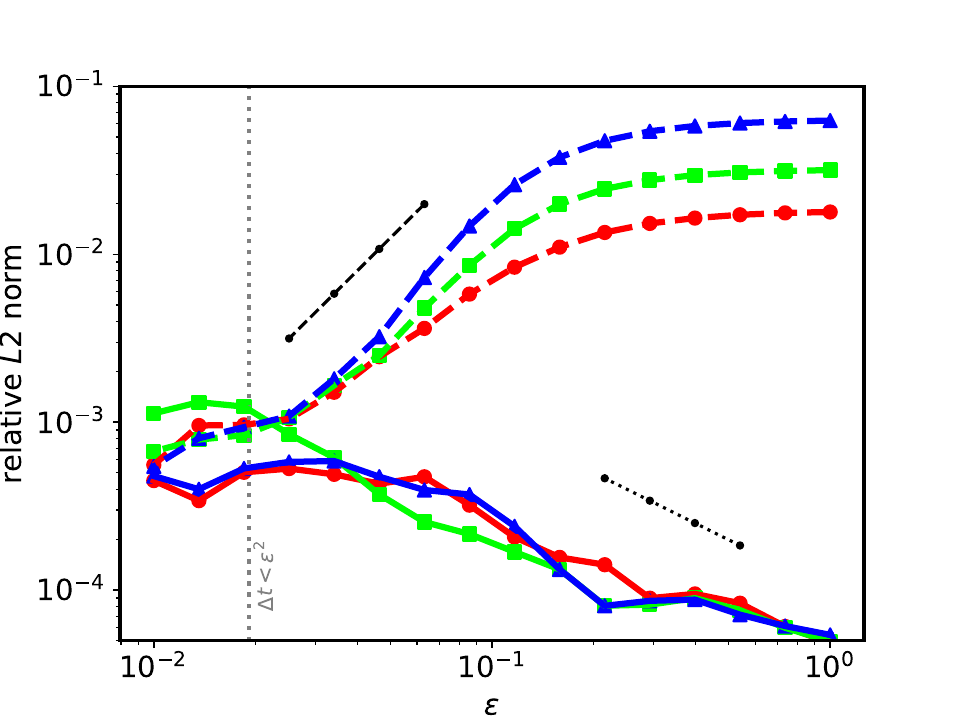}
        \caption{Fixed $\dt=\bar{t}/75$, varying $\varepsilon$}
        \label{fig: est error n=75}
    \end{subfigure}

    \captionsetup{justification=centerlast}
    \caption{The total estimation error $\varepsilon_{kd}$ as $\varepsilon\rightarrow0$. The plots are arranged in row-major order with $1, 5, 25, $ and $75$ time segments.}
    \label{fig: est error e}
\end{figure}

As for the time step $\dt$, the total estimation error $\epsilon_{kd}(\dt) = \BO(\dt)$ when $\dt\ll\varepsilon^2$, and $\epsilon_{kd}(\dt) = \BO(1/\dt)$ when $\dt\gg\varepsilon^2$. 
The convergence is displayed in Figure \ref{fig: est error dt}, with fixed values of $\varepsilon = 0.1$, $0.05$, $0.02$, and $0.01$ shown in Figures \ref{fig: est error e=0.1}–\ref{fig: est error e=0.01}, respectively.
The dashed lines are flat, indicating that the error of the fluid method is independent of the time step.
When $\dt\ll\varepsilon^2$, such as the left side of the vertical line in Figures \ref{fig: est error e=0.1} and \ref{fig: est error e=0.05}, the KDMC error converges with the order $\BO(\dt)$. 
When $\dt\gg\varepsilon^2$, such as the right side of the vertical line in Figures~\ref{fig: est error e=0.05} and ~\ref{fig: est error e=0.02}, this error grows with the order $\BO(1/\dt)$. 
Note that this convergence is not immediately apparent, as the KDMC error is already low and close to the reference error level in this homogeneous test case, thereby obscuring the convergence. 
This convergence can be observed more clearly in the fusion-relevant heterogeneous test case in Section \ref{sec: fusion test case}.
Finally, when $\varepsilon=0.01$, shown in Figure \ref{fig: est error e=0.01}, KDMC converges to the fluid model. Thus, the errors of both methods are close.  

In conclusion, this experiment confirms our expectation that KDMC is more accurate than the fluid method, 
except in the high-collisional regime, where KDMC converges to the fluid solution and both methods exhibit similar levels of accuracy.

\begin{figure}[h]
    \centering
    \begin{subfigure}[b]{0.45\textwidth}
        \centering
        \includegraphics[trim=0 0 0 20, width=\textwidth]{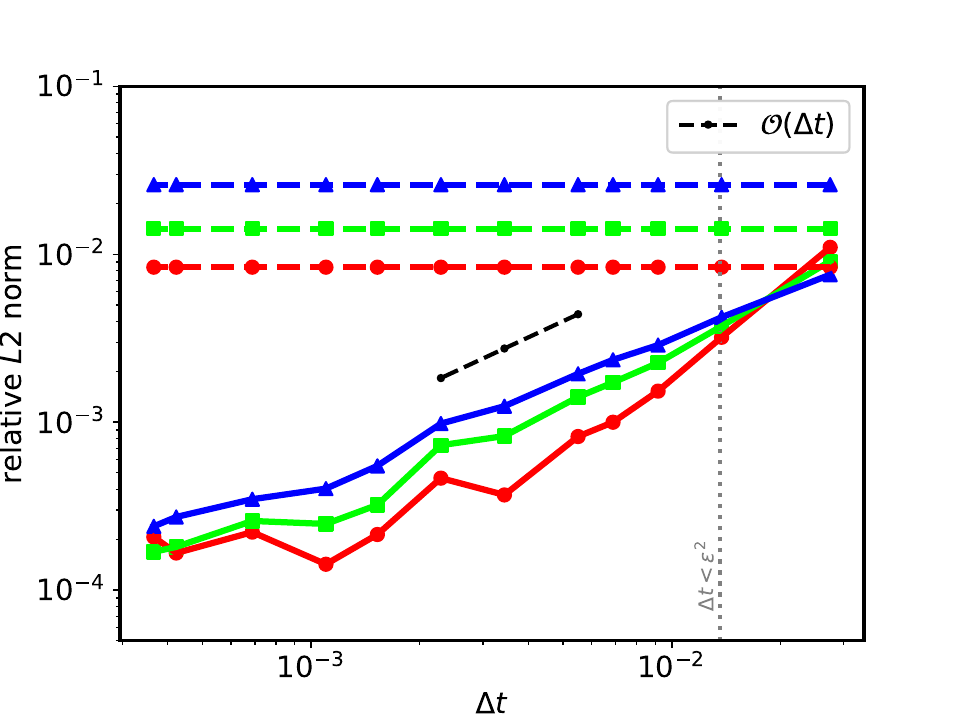}
        \caption{Fixed $\varepsilon\approx0.1$, varying $\dt$}
        \label{fig: est error e=0.1}
    \end{subfigure}
    \begin{subfigure}[b]{0.45\textwidth}
        \centering
        \includegraphics[trim=0 0 0 30, width=\textwidth]{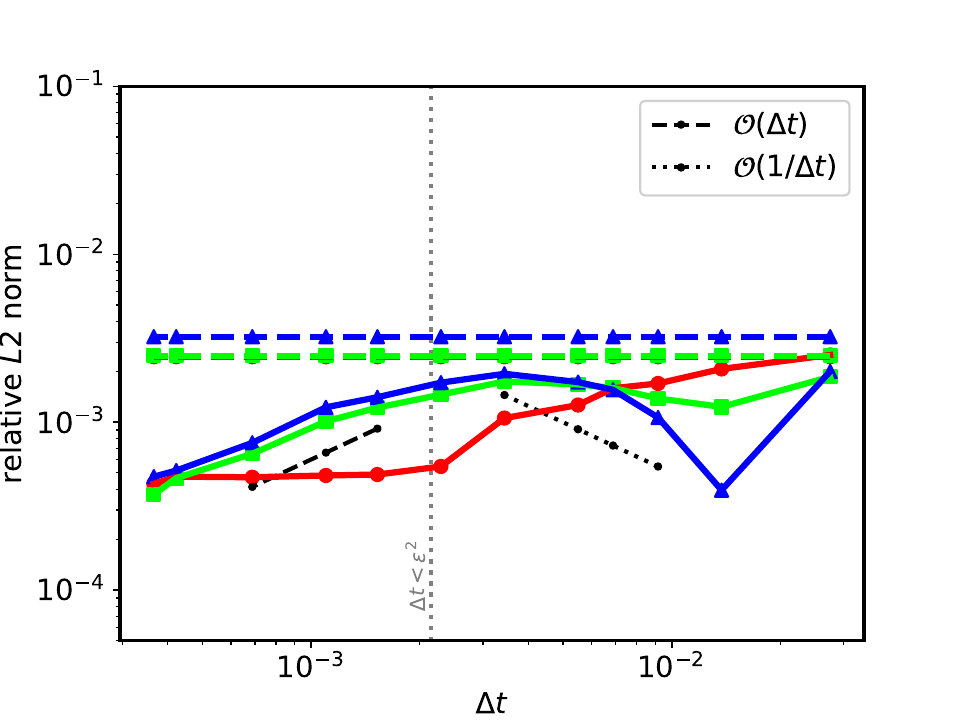}
        \caption{Fixed $\varepsilon\approx0.05$, varying $\dt$}
        \label{fig: est error e=0.05}
    \end{subfigure}

    \begin{subfigure}[b]{0.45\textwidth}
        \centering
        \includegraphics[width=\textwidth]{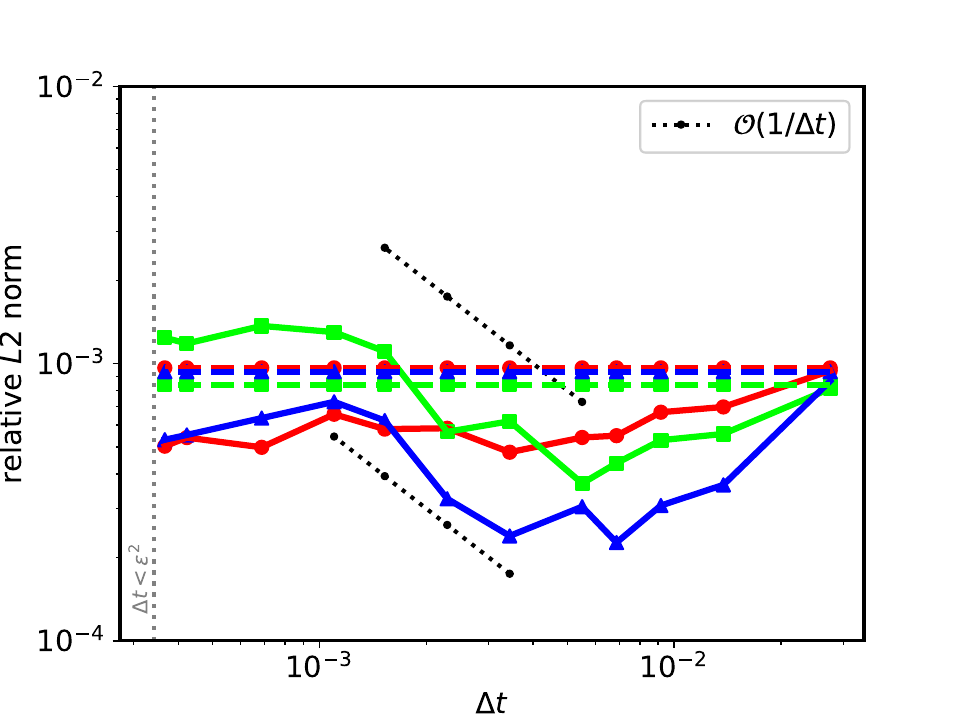}
        \caption{Fixed $\varepsilon\approx0.02$, varying $\dt$}
        \label{fig: est error e=0.02}
    \end{subfigure}
    \begin{subfigure}[b]{0.45\textwidth}
        \centering
        \includegraphics[width=\textwidth]{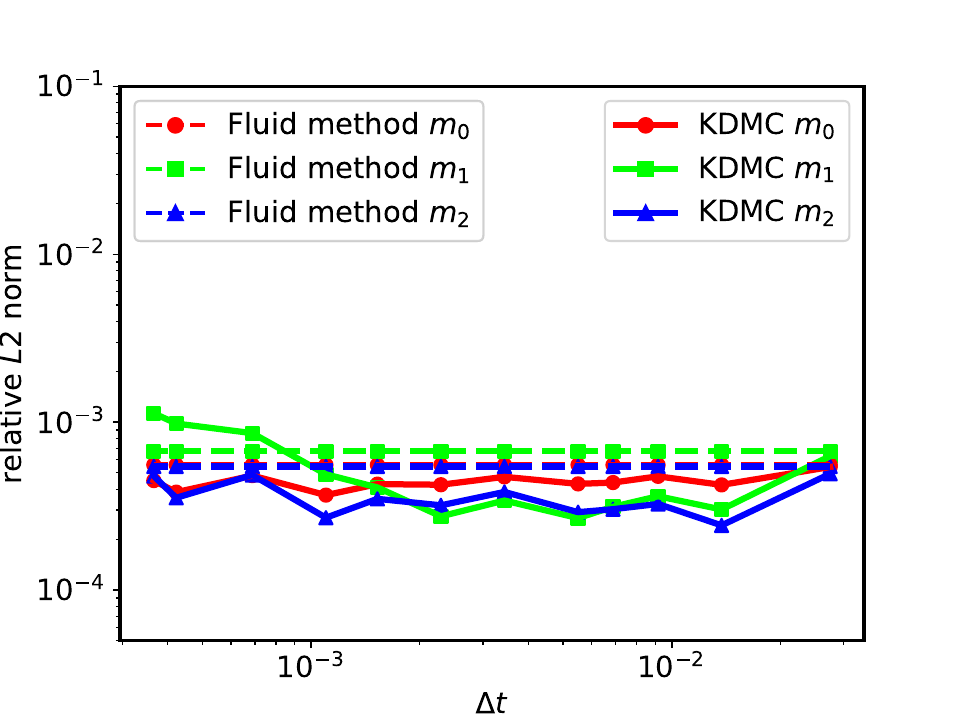}
        \caption{Fixed $\varepsilon=0.01$, varying $\dt$}
        \label{fig: est error e=0.01}
    \end{subfigure}

    \captionsetup{justification=centerlast}
    \caption{Total estimation error $\epsilon_{kd}$ as $\dt \to 0$. The plots are arranged in row-major order for $\varepsilon \approx 0.1$, $0.05$, $0.02$, and $0.01$. In subplot (b), the expected convergence rate of $\BO(1/\dt)$ is less apparent, as the KDMC error is already low and close to the reference error level. Clear convergence behavior is observed in the fusion-relevant test case, as predicted by our theory.}
    \label{fig: est error dt}
\end{figure}

\subsection{Fusion-relevant heterogeneous test case}\label{sec: fusion test case}
The numerical experiments presented earlier are based on a homogeneous background. 
However, as explained in Section \ref{sec: ana estimation error}, if the moderate assumption $R(x)=\BO(1/\varepsilon^2)$ is satisfied, the total estimation error \eqref{eqn: case e_kd} and further \eqref{eqn: err kd dt fixed} also hold in a heterogeneous background.  
In this section, we demonstrate the effectiveness of KDMC and our analysis for the heterogeneous background using a fusion-relevant test case.

In the plasma edge region of a Tokamak, particles interact with the solid wall and undergo reflection, necessitating the use of reflective boundary conditions. However, the study of such boundary conditions within the KDMC framework is still an open research
question, and we leave this question for future work. Therefore, we consider a test case that adopts the profile from a realistic fusion scenario \cite{horstenComparisonFluidNeutral2016c}, but employs the periodic boundary condition instead.

In this test case, the particle system is simulated on the periodic domain $x\in[0, 1]$, from $t=0$ s to $\bar{t}=0.001$ s. The initial density is given by $\rho(x,t=0) = 1 + 1/(2\pi)\times\sin(2\pi x)$ [m$^{-1}$]. 
The variance of the plasma velocity and the collision rate are set to the values used in \cite{horstenComparisonFluidNeutral2016c}, namely, 
\begin{equation}\label{eqn: hetero bg}
    	\sigma_p^2(x) = \frac{eT_i(x)}{m_p} \quad [\mathrm{m}^2\mathrm{s}^{-2}], \quad \text{and} \quad R(x) =  \rho_i \cdot 3.2\times10^{-15}\left(\frac{T_i(x)}{0.026}\right)^{1/2} \quad [\mathrm{s}^{-1}], 
\end{equation}
where $e\approx 1.60\times10^{-19}$~J$\cdot$eV$^{-1}$ is the value of one electronvolt, $m_p\approx 1.67 \times 10^{-27}$~kg is the plasma particle mass, 
and $T_i(x) = 5.5 + 4.5\times \cos(2\pi x)$~[eV] is the temperature of ions, 
which reaches $10\,\text{eV}$ at the domain boundaries and equals $1\,\text{eV}$ at the center of the domain. 
The meaning and units of the remaining constants can be found in \cite{lovbakMultilevelAdjointMonte2023}.
The ion density $\rho_i = 10^{21}$ m$^{-1}$ is a constant. 
The background \eqref{eqn: hetero bg} is modeled as being close to the limit in a diffusive scaling. To make this modeling assumption explicit, we introduce the scaled quantities
\begin{equation}
    \tilde\sigma_p^2(x) = \frac{10^{-7}}{\varepsilon^2}\sigma_p^2, \quad \text{and} \quad \tilde R(x) = \frac{10^{-7}}{\varepsilon^2} R(x),
\end{equation}
given that $\sigma_p^2(x)$ and $R(x)$ have an order of magnitude of $10^7$. In the experiment, we vary $\varepsilon$ from $10^0$ to $10^{-3.5}$. 
Lastly, the mean plasma velocity is chosen as $ u_p(x) = 100 + 1/(6\pi)\times\sin(6\pi x)$~[ms$^{-1}$]. 

Using the kinetic method as a reference and the fluid method as the comparison, the convergence is displayed in Figure \ref{fig: hetero est error}. 
The same color and line conventions apply as in Section \ref{sec: estimation error}. 
The point where the convergence changes in a homogeneous background is around $\dt=1/R\approx\varepsilon^2$. 
In this heterogeneous background, we choose this point as $\dt=\text{mean}(1/\tilde{R}(x))$, the mean of $1/\tilde{R}(x)$, indicated by the vertical dashed line.

Fixing the time step size $\dt=\Bar{t}/85$, the error against the diffusive scaling parameter $\varepsilon$ is shown in Figure \ref{fig: hetero est error e}.
In this test case, the error of the fluid method (dashed lines) starts to converge only when $\varepsilon$ is around $10^{-3}$, while KDMC  (solid lines) has smaller errors for all $\varepsilon$. 
We see again, as in the homogeneous case in Section \ref{sec: estimation error}, 
that the total KDMC estimation error (solid lines) $\epsilon_{kd}(\varepsilon) = \BO(1/\varepsilon)$ when $\dt\ll\varepsilon^2$, and $\epsilon_{kd}(\varepsilon) = \BO(\varepsilon^2)$ when $\dt\gg\varepsilon^2$ as predicted in \eqref{eqn: err kd dt fixed}. 
When $\varepsilon=10^{-3.5}$, which is the fusion test case in \cite{horstenComparisonFluidNeutral2016c}, KDMC outperforms the fluid model in terms of error by one order of magnitude. 
In the non-high collisional regime, KDMC could be even more accurate.  
Note that the specific choice of the time step size $\dt$ is made because the results are easier to interpret visually. 
The same convergence can be observed with other time step sizes (see more experiments at \cite{tangPythonCodeThis2025}). 

Fixing \( \varepsilon = 0.0072 \), the error varying with \( \Delta t \) is illustrated in Figure~\ref{fig: hetero est error dt}. The analytical convergence rate from \eqref{eqn: err kd dt fixed} is clearly observable: when \( \Delta t \ll \varepsilon^2 \),  \(\epsilon_{kd}(\dt) = \mathcal{O}(\Delta t) \); whereas when \( \Delta t \gg \varepsilon^2 \), \( \epsilon_{kd}(\dt)=\mathcal{O}(1/\Delta t) \). 
For all choices of $\dt$, the error of KDMC is lower than that of the fluid method.
Especially for $\dt\geq10^{-4}$, the error of KDMC is more than one order of magnitude lower than that of the fluid method.

In conclusion, the performance of KDMC and the associated fluid estimation follows our analysis in Section \ref{sec: estimation sec}. 
In the heterogeneous case, the analyzed algorithm outperforms the fluid method even in the high-collision regime.  

 \begin{figure}[h]
 \centering
    \makebox[\textwidth]{
     \begin{subfigure}{0.45\textwidth}
         \centering
         \includegraphics[trim=0 0 0 30, clip, width=\textwidth]{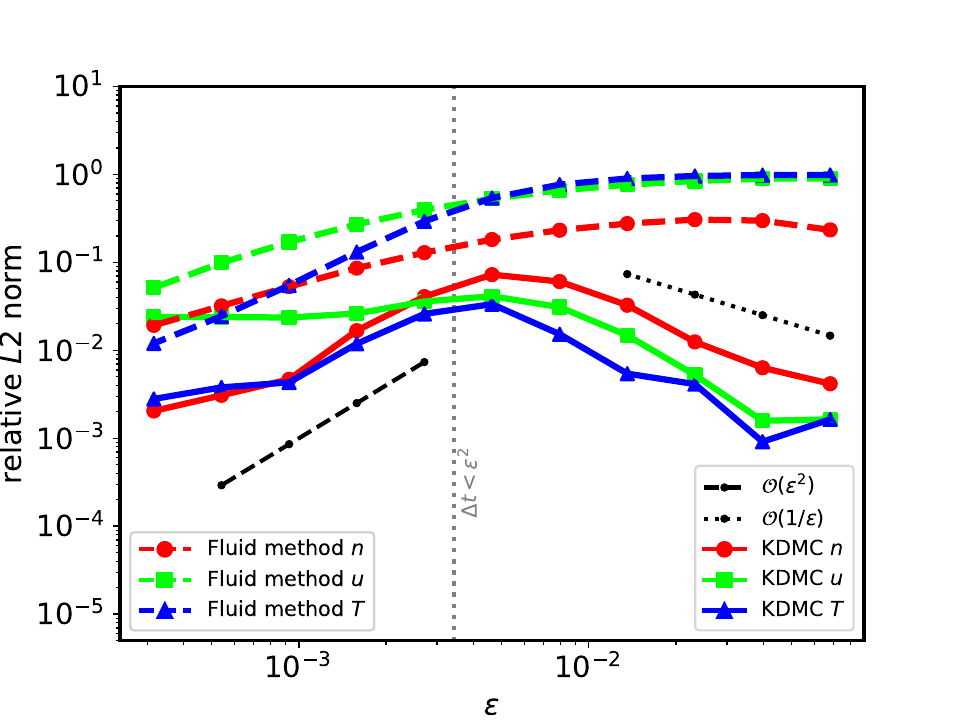}
         \caption{Fixed $\dt=\Bar{t}/85$, varying $\varepsilon$}
         \label{fig: hetero est error e}
     \end{subfigure}
     \hspace{0.02\textwidth}
     \begin{subfigure}{0.45\textwidth}
         \centering
         \includegraphics[trim=0 0 0 30, clip, width=\textwidth]{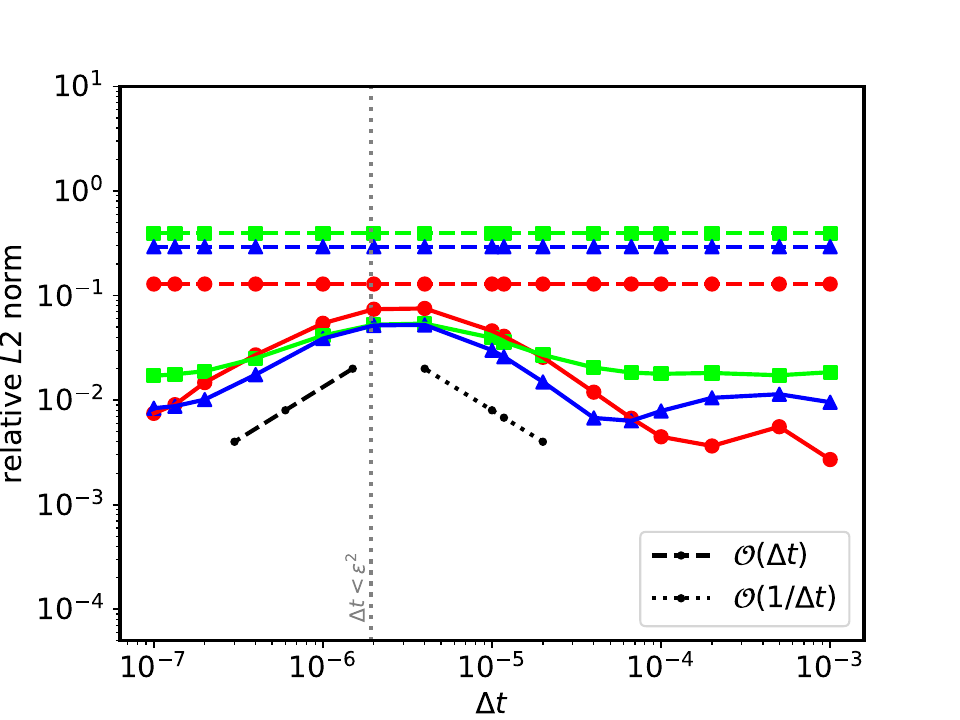}
         \caption{Fixed $\varepsilon=0.0027$, varying $\dt$}
         \label{fig: hetero est error dt}
     \end{subfigure}
     }
     \captionsetup{justification=centerlast}
     \caption{The total estimation error $\epsilon_{kd}$ of the fusion-relevant test case.}
     \label{fig: hetero est error}
\end{figure}

\subsection{Computational cost of KDMC}\label{sec: computational cost}
The KDMC method is expected to offer higher accuracy than the fluid model while being computationally more efficient than the kinetic method. The accuracy is demonstrated via the homogeneous test case in Section \ref{sec: estimation error} and the fusion-relevant test case in Section~\ref{sec: fusion test case}. In this section, we discuss the computational cost.

We begin by analyzing the simulation stage when the trajectories of particles are generated. 
We assume the computational cost of a kinetic step~\eqref{eqn: a kinetic step} and a diffusive step~\eqref{eqn: SDE update} is comparable, as both steps first generate one or two random numbers and update the particle’s position, as shown in Algorithm~\ref{alg: KDMC with f}.
Therefore, we treat each step as a unit computational operation.
The kinetic simulation updates the position along with each collision. Thus, if we assume the collision rate $R=1/\varepsilon^2$, the number of operations, denoted as $C_k$, is $C_k = R\Bar{t}=\Bar{t}/\varepsilon^2$ with $\Bar{t}$ the simulation time and $R\Bar{t}$ the total number of collisions.
KDMC has two positional updates in each $\dt$; hence, the number of operations, denoted as $C_{kd}$, is $C_{kd} = 2\Bar{t}/\dt=2K$, where $K$ is the number of time steps of size $\dt$.
The KDMC simulation has fewer operations than the kinetic simulation if $C_{kd}<C_k$, which gives the condition $\dt>2\varepsilon^2$. 
It indicates roughly the regime where the error of the diffusive part dominates. For instance, the left side of the vertical dashed line in Figure \ref{fig: hetero est error e}, and the right side of it in Figure \ref{fig: hetero est error dt}. 

When the estimation stage is included, moments in the kinetic method can be estimated concurrently with the simulation (see Algorithm \ref{alg: KDMC with f}), incurring negligible additional cost.
As for KDMC, the fluid estimation is performed after the simulation. This estimation solves the fluid model \eqref{eqn: fluid-model} with the evolution time $\hat\theta$ bounded by the time step $\dt$. 
In the meantime, the computational cost of solving the fluid model in the estimation stage is typically negligible compared with the cost of the KDMC simulation.
Therefore, the computational cost is decided by the simulation stage in both the kinetic method and KDMC, and the speed-up achieved by using KDMC instead of the kinetic method is effectively quantified by the ratio of operations $C_k/C_{kd}\approx \dt/\varepsilon^2$.

Table~\ref{tab: timing} presents the computation time for the fusion-relevant test case introduced in Section~\ref{sec: fusion test case}, with $\varepsilon=0.0027$, corresponding to Figure~\ref{fig: hetero est error dt}.
The first three columns display the number of time steps, the simulation time, and the estimation time measured in seconds, respectively.  It is evident that the estimation time is negligible compared to the simulation time. 
The fourth column presents the speed-up, compared to the run-time of the reference kinetic method, which is $48066.03$ seconds. 
It can be readily verified that the speed-up is proportional to $1/K$ with $K$ the number of time steps. 
As $K$ increases such that $C_{kd} > C_k$, that is, $\Delta t = \bar{t}/K \leq 2\varepsilon^2$, KDMC becomes less efficient, resulting in a slowdown indicated in the gray-highlighted rows of the table.

\begin{table}[h]
\centering
\caption{Runtime of KDMC simulation with fluid estimation.  Speed-up is computed by comparing the total KDMC runtime to that of the kinetic method (\(48066.03\) seconds). Cases with a slowdown are highlighted in gray. }
\begin{tabular}{c|c|c|c}
\hline
\#time steps $K$ & Simulation time ($s$) & Estimation time ($s$) & Speed-up \\
\hline
1   & 1216.29  & 0.111105  & 39.47 \\
2   & 1874.68  & 0.052020  & 25.64 \\
5   & 3637.70  & 0.020168  & 13.22 \\
10  & 6686.51  & 0.011620  & 7.19 \\
15  & 9701.11  & 0.006695  & 4.95 \\
25  & 15758.76 & 0.004488  & 3.05 \\
50  & 30865.44 & 0.002064  & 1.56 \\
\rowcolor{gray!20}
85  & 51739.55  & 0.001479  & 0.93 \\
\rowcolor{gray!20}
100 & 60388.68  & 0.001100  & 0.80 \\
\hline
\end{tabular}

\label{tab: timing}
\end{table}
%

We finish this section by remarking that the test case above is performed in a one-dimensional problem in space. 
If a two- or three-dimensional case is considered, the computational cost of solving the fluid model in the estimation stage may be significant if a large $\dt$ is used. 
Nevertheless, we expect that the cost of the KDMC simulation will still dominate. 
The analysis presented in this work can guide users in choosing an appropriate time step $\dt$ for different application scenarios.

\section{Conclusion}\label{sec: conclusion}
In this work, we analyze the KDMC method combined with a fluid estimation procedure. A fully kinetic Monte Carlo simulation is used as a reference, and a purely fluid model serves as a comparison. 
In addition to the asymptotic analysis, we conduct numerical experiments to verify the theoretical results. Homogeneous test cases are employed to assess each individual error, as well as the total error.
Moreover, a fusion-relevant heterogeneous test case validates our analysis and the effectiveness of the analyzed algorithm for application-relevant problems.

In the first part of the analysis, we derive an error bound for the KDMC simulation based on the 1-Wasserstein distance. 
The error expression reveals the property that the KDMC simulation error remains low for both small and large diffusive scaling parameter $\varepsilon$ (corresponding to the high- and low-collision regimes, respectively), but exhibits a relatively large bias in intermediate regimes. 
A similar property holds for the time step $\dt$: the error is low when $\dt$ is either small or large, but relatively high in between. 
In particular, the bias reaches its peak around $\dt=\varepsilon^2$.

In the second part of the analysis, the estimation error is discussed. 
The estimation includes the kinetic part and the diffusive part.  
For the kinetic part, an unbiased MC estimator is used with only a statistical error. 
Besides the number of particles, this error is also affected by the time step size $\dt$, and the diffusive scaling parameter $\varepsilon$.  
We also observe that when the mean velocity of the background plasma is small, the statistical error of the $m_1$ moment is larger than that of the $m_0$ and $m_2$ moments. 
For the diffusive part, the fluid estimation procedure is applied, which estimates the moments by solving the corresponding fluid model. 
The initial condition is constructed from the collection of all initial positions of the diffusive steps, while the evolution time is the average duration of all of these steps. 
This procedure introduces several sources of error, including stochastic and deterministic errors, each of which is analyzed both analytically and numerically. 

To account for the mutual influence between the kinetic and diffusive parts, as well as the stochastic effect, dedicated experiments are designed to isolate and illustrate each individual error, confirming the validity of our analysis.
As expected, KDMC is more accurate than the purely fluid method. Notably, in the fusion case, KDMC achieves an order of magnitude higher accuracy than the fluid method.
Finally, the computational cost analysis shows that 
KDMC achieves a speedup over the kinetic method of a factor of $\dt/\varepsilon^2$, given that the cost of the fluid estimation is negligible.

This work provides a detailed analysis of KDMC and its associated fluid estimation. In future work, we plan to develop a multilevel scheme for the analyzed algorithm to accelerate the convergence. 
Moreover, we intend to derive a more accurate fluid model and integrate it into the KDMC framework to improve overall accuracy.

\appendix
\section{Pseudo-time simulation}\label{appendix: steady-state solution}
We compare two ways to simulate the steady-state Boltzmann-BGK equation
\begin{equation} \label{eqn: steady-state ss eq}
    v\partial_x \phi (x,v)= -R_i(x) \phi(x,v) + R_{cx}(x)\left(\intv \phi(x,v') dv' \mathcal{M}(v) - \phi(x,v)\right) + S(x,v),
\end{equation}
by the time-dependent equation with the standard MC method, where the distribution $\phi(x,v)$ is the steady-state solution. 
$R_i(x)$ and $R_{cx}(x)$ are the ionization and charge-exchange collision rates, respectively. $S(x,v)$ is the source.

\subsection{Terminal-time simulation}\label{sec: terminal simulation}
The first way follows the pseudo-time stepping method \cite{kelleyConvergenceAnalysisPseudoTransient1998b}. Consider the time-dependent equation
\begin{equation}\label{eqn: A1 bgk}
    \partial_t f(x,v,t) + v\partial_x f(x,v,t) = -R_i(x) f(x,v,t) + R_{cx}(x)\left(\mathcal{M}(v)\intv f(x,v',t) dv'  - f(x,v,t)\right) + S(x,v),
\end{equation}
with the unknown $f(x,v,t)$ and an initial guess $f(x,v,t=0)$. The solution at a sufficiently large time $t\approx\infty$, denoted as $f(x,v,\infty)$, is the steady-state solution, that is, $f(x,v,\infty)=\phi(x,v)$
the solution of \eqref{eqn: steady-state ss eq}.
When the standard MC method is employed, we simulate the time-dependent equation \eqref{eqn: A1 bgk} up to a sufficiently large terminal time $\Bar{t}$, 
and estimate the steady-state distribution $f(x,v,\Bar{t})$ by recording the number of particles in each phase-space cell at time $\Bar{t}$.

\subsection{Time-integrated simulation}
Alternatively, we can define $f(x,v,t)$ by
\begin{equation*}
    \int_0^\infty f(x,v,t) dt = \phi(x,v),
\end{equation*}
and solve 
\begin{equation}\label{eqn: steady-state phi}
    \partial_t f(x,v,t) + v\partial_x f(x,v,t) = -R_i(x) f(x,v,t) + R_{cx}(x)\left(\mathcal{M}(v) \intv f(x,v',t) dv'  - f(x,v,t) \right), 
\end{equation}
with the initial condition $f(x,v,t=0)=S(x,v)$ and the property
\begin{equation}\label{eqn: steady-state phi initial condition}
    f(x,v,t=\infty)=0,
\end{equation} 
due to the sink term $-R_i(x) f(x,v,t)$.
Integrating \eqref{eqn: steady-state phi} over time $t$ from $0$ to $\infty$, we have
\begin{equation*}
    f(x,v,t=\infty) - f(x,v,t=0) + v\partial_x \phi(x,v) = -R_i \phi(x,v) + R_{cx}(\mathcal{M}(v) \intv \phi(x,v) dv  - \phi(x,v)).
\end{equation*}
Substituting the initial condition $S(x,v)$ and the property \eqref{eqn: steady-state phi initial condition}, we obtain
\begin{equation}
    v\partial_x \phi (x,v)= -R_i(x) \phi(x,v) + R_{cx}(x)\left(\intv \phi(x,v') dv' \mathcal{M}(v) - \phi(x,v)\right) + S(x,v),
\end{equation}
which is exactly the steady-state equation \eqref{eqn: steady-state ss eq} we want to solve. Hence, we can get $\phi(x,v)$ by integrating $f(x,v,t)$, the solution of \eqref{eqn: steady-state phi}, over time $t$.
When using the standard MC, we again simulate the time-dependent equation up to a sufficiently large time $\Bar{t}$, 
but in this approach, we record the number of particles in each phase-space cell at all times from $t=0$ to $\Bar{t}$. 
Compared with the MC method in \ref{sec: terminal simulation}, which only records the number of particles at time $\Bar{t}$, the MC method in this section scores more samples and therefore achieves a lower variance (or error).
Moreover, although the latter improves statistical accuracy, it has almost the same computational cost as the former. 

For the purpose of algorithmic analysis, we neglect the ionization term $R_i(x)\phi(x,v)$, since KDMC is designed for systems involving charge-exchange collisions only. 
The ionization process can instead be handled separately using an operator-splitting approach \cite{macnamaraOperatorSplitting2016b, maesHilbertExpansionBased2023a}. As a result, the time-dependent equation \eqref{eqn: steady-state phi} reduces to the Boltzmann–BGK equation \eqref{eqn: BBGK} considered in this work.
\section{Mild formulation of Boltzmann-BGK equation}\label{appendix: mild formulation}
In this section, we derive the mild formulation \cite{dipernaCauchyProblemBoltzmann1989} of the Boltzmann-BGK equation \eqref{eqn: BBGK}. 
The Boltzmann-BGK equation  without the ionization term \eqref{eqn: BBGK} reads
\begin{equation}\label{eqn: BBGK no ionization}
\partial_t f(x,v,t) + v\partial_x f(x,v, t) = R\left(M[f](x,v,t)-f(x,v,t)\right),    
\end{equation}
with the initial condition
\begin{equation*}
    f(x,v,t=0) = f_0(x,v),
\end{equation*}
where 
\begin{equation}\label{eqn: m[f]}
    M[f](x,v,t) = M_p(v|x)\intv f(x,v', t)dv'.
\end{equation}
and $R$ the constant collision rate, since we consider a homogeneous background. Using the method of characteristics, Equation \eqref{eqn: BBGK no ionization} can be written as 
\begin{equation}\label{eqn: ODE}
    \frac{d }{d t} f(x(t), v, t) = R\left(M[f](x(t),v) - f(x(t),v)\right),
\end{equation}
along the path 
\begin{equation}\label{eqn: path}
    x(t) = x(s) + v\cdot(t-s),
\end{equation}
where $0\leq s \leq t$ and we have
\begin{equation}\label{eqn: ode initial}
    f(x(0),v,t=0) = f_0(x(t)-vt, v).
\end{equation}
We solve the ODE \eqref{eqn: ODE} by first multiplying the factor $e^{tR}$ on both side of \eqref{eqn: ODE} which gives
\begin{equation*}
    e^{Rt}\frac{d }{d t} f(x(t), v, t) +Re^{Rt}f(x(t),v)= Re^{Rt}M[f](x(t),v, t).
\end{equation*}
Equivalently, we have 
\begin{equation}\label{eqn: ddt}
    \frac{d}{dt}\left(e^{Rt}f(x(t), v, t) \right)= Re^{Rt}M[f](x(t),v,t).
\end{equation}
Then, integrating \eqref{eqn: ddt} over time from $0$ up to $t$ gives
\begin{equation}\label{eqn: mild_1}
    e^{Rt}f(x(t), v, t) = f(x(0), v, 0) + \int_0^t Re^{Rs}M[f](x(s),v,s)ds.
\end{equation}
Substituting \eqref{eqn: m[f]}, \eqref{eqn: path} and \eqref{eqn: ode initial} into \eqref{eqn: mild_1}, and rearranging terms, we obtain 
\begin{equation}
    f(x,v,t) = e^{-Rt}f_0(x-vt, v) + \int_0^t Re^{-R(t-s)}M(v|x-v(t-s))\int_{-\infty}^{\infty} f(x-v'(t-s), v', s) dv' ds,
\end{equation}
the mild formulation of the Boltzmann-BGK equation without the ionization term \eqref{eqn: BBGK no ionization}. Note that if the heterogeneous background $R=R(x)$ is considered, the derivation above still holds. We will have
\begin{equation*}
    f(x,v,t) = e^{\Lambda(t,0;x,v)}f_0(x-vt, v) + \int_0^t R(x-v(t-s))e^{\Lambda(t,s;x,v)}M[f](v|x-v(t-s))ds,
\end{equation*}
with 
\begin{equation*}
    \Lambda(t,0;x,v) = \int_s^t R(x-v(t-\tau))d\tau.
\end{equation*}

\section*{Acknowledgements}
This work has been carried out within the framework of the EUROfusion Consortium, funded by the European Union via the Euratom Research and Training Programme (Grant Agreement No 101052200 — EUROfusion). The views and opinions expressed herein do not necessarily reflect those of the European Commission.
Part of this research was funded by the Research Foundation Flanders (FWO) under grant G085922N.
Emil Løvbak was funded by the Deutsche Forschungsgemeinschaft (DFG, German Research Foundation) – Project-ID 563450842. 
The authors are also grateful to Vince Maes for valuable discussions.

\section*{CRediT authorship contribution statement}
\textbf{Zhirui Tang:} Conceptualization, Methodology, Software, Investigation, Visualization, Writing – Original Draft.  
\textbf{Julian Koellermeier:} Supervision, Writing – Review \& Editing.  
\textbf{Emil Løvbak:} Supervision, Methodology, Validation, Writing – Review \& Editing.  
\textbf{Giovanni Samaey:} Conceptualization, Supervision, Funding Acquisition, Writing – Review \& Editing.

\section*{Data availability}
The data and source code that support the findings of this study are publicly available at:  
\url{https://gitlab.kuleuven.be/u0158749/kdmc_and_fluid_estimation_analysis}

\bibliographystyle{elsarticle-num-names} 
\bibliography{Fluid_Estimator}

\end{document}